\RequirePackage{fix-cm}
\documentclass[smallextended]{svjour3}       
\smartqed  
\usepackage{graphicx}
\usepackage{latexsym,amsfonts,amsmath,amssymb,mathrsfs,url}
\usepackage{siunitx}
\newcommand{\rd}{\, \mathrm{d}}
\newcommand{\ri}{\mathrm{i}}
\newcommand{\bszero}{\boldsymbol{0}}
\newcommand{\bsh}{\boldsymbol{h}}
\newcommand{\bsk}{\boldsymbol{k}}
\newcommand{\bsl}{\boldsymbol{\ell}}
\newcommand{\bsx}{\boldsymbol{x}}
\newcommand{\bsy}{\boldsymbol{y}}
\newcommand{\bsz}{\boldsymbol{z}}
\newcommand{\bsalpha}{\boldsymbol{\alpha}}
\newcommand{\bssigma}{\boldsymbol{\sigma}}
\newcommand{\bsgamma}{\boldsymbol{\gamma}}
\newcommand{\bstau}{\boldsymbol{\tau}}
\newcommand{\Hcal}{\mathcal{H}}
\newcommand{\NN}{\mathbb{N}}
\newcommand{\RR}{\mathbb{R}}
\newcommand{\ZZ}{\mathbb{Z}}
\newcommand{\rsob}{\mathrm{sob}}
\newcommand{\rcos}{\mathrm{cos}}
\newcommand{\rkor}{\mathrm{kor}}
\newcommand{\wor}{\mathrm{wor}}
\newcommand{\odd}{\mathrm{odd}}
\newcommand{\even}{\mathrm{even}}
\newcommand{\rsobodd}{\mathrm{sob}(\mathrm{odd})}
\newcommand{\rsobodda}{\mathrm{sob}(\mathrm{odd}+\alpha)}
\newcommand{\rsoboddbdry}{\mathrm{sob}(\mathrm{odd-bdry}0)}
\newcommand{\rsym}{\mathrm{sym}}
\spnewtheorem*{xproof}{}{\itshape}{\rmfamily}

\renewenvironment{proof}[1][\proofname]
 {\renewcommand\xproofname{#1}\xproof}
 {\endxproof}
\allowdisplaybreaks
\begin{document}

\title{Lattice rules in non-periodic subspaces of \\ Sobolev spaces}


\author{Takashi Goda   \and Kosuke Suzuki \and Takehito Yoshiki
}


\institute{T. Goda \at
School of Engineering, University of Tokyo, 
7-3-1 Hongo, Bunkyo-ku, Tokyo 113-8656, Japan \\
\email{goda@frcer.t.u-tokyo.ac.jp}
           \and
           K. Suzuki \at
Graduate School of Science, Hiroshima University, 
1-3-1 Kagamiyama, Higashi-Hiroshima, 739-8526, Japan. 
JSPS Research Fellow \\
\email{kosuke-suzuki@hiroshima-u.ac.jp}
          \and
          T. Yoshiki \at
Department of Applied Mathematics and Physics, Graduate School of Informatics, Kyoto University, 
Kyoto 606-8561, Japan \\
\email{yoshiki.takehito.47x@st.kyoto-u.ac.jp}}

\date{Received: date / Accepted: date}

\maketitle

\begin{abstract}
We investigate quasi-Monte Carlo (QMC) integration over the $s$-dimensional unit cube based on rank-1 lattice point sets in weighted non-periodic Sobolev spaces $\Hcal(K_{\alpha,\bsgamma,s}^{\rsob})$ and their subspaces of high order smoothness $\alpha>1$, where $\bsgamma$ denotes a set of the weights. A recent paper by Dick, Nuyens and Pillichshammer has studied QMC integration in half-period cosine spaces with smoothness parameter $\alpha>1/2$ consisting of non-periodic smooth functions, denoted by $\Hcal(K_{\alpha,\bsgamma,s}^{\rcos})$, and also in the sum of half-period cosine spaces and Korobov spaces with common parameter $\alpha$, denoted by $\Hcal(K_{\alpha,\bsgamma,s}^{\rkor+\rcos})$. Motivated by the results shown there, we first study embeddings and norm equivalences on those function spaces. In particular, for an integer $\alpha$, we provide their corresponding norm-equivalent subspaces of $\Hcal(K_{\alpha,\bsgamma,s}^{\rsob})$. This implies that $\Hcal(K_{\alpha,\bsgamma,s}^{\rkor+\rcos})$ is strictly smaller than $\Hcal(K_{\alpha,\bsgamma,s}^{\rsob})$ as sets for $\alpha \geq 2$, which solves an open problem by Dick, Nuyens and Pillichshammer. Then we study the worst-case error of tent-transformed lattice rules in $\Hcal(K_{2,\bsgamma,s}^{\rsob})$ and also the worst-case error of symmetrized lattice rules in an intermediate space between $\Hcal(K_{\alpha,\bsgamma,s}^{\rkor+\rcos})$ and $\Hcal(K_{\alpha,\bsgamma,s}^{\rsob})$. We show that the almost optimal rate of convergence can be achieved for both cases, while a weak dependence of the worst-case error bound on the dimension can be obtained for the former case.
\end{abstract}

\section{Introduction}
In this paper we study multivariate integration of smooth functions defined over the $s$-dimensional unit cube $[0,1]^s$.
For an integrable function $f\colon [0,1]^s \to \RR$, we denote the integral of $f$ by
\[ I(f) = \int_{[0,1]^s}f(\bsx)\rd \bsx. \]
A quasi-Monte Carlo (QMC) rule denotes an approximation of $I(f)$ by the average of function evaluations on a finite point set $P\subset [0,1]^s$:
\[ Q_{P}(f) = \frac{1}{|P|}\sum_{\bsx\in P}f(\bsx), \]
where we interpret $P$ as a set in which the multiplicity of elements matters.
For a function space $V$ with norm $\|\cdot\|_V$, the worst-case error of a QMC rule using a point set $P\subset [0,1]^s$ is defined by
\[ e^{\wor}(P;V) := \sup_{\substack{f\in V\\ \|f\|_V\leq 1}}\left| Q_{P}(f)-I(f) \right|. \]
The aim of this paper is to construct a good deterministic point set $P$ which makes the worst-case error small for a specific function space.

There are two main families for QMC point sets: digital nets and integer lattices.
We refer the reader to \cite{DKS13,DPbook,Nbook,SJbook} and the references cited therein for general information on this subject.
In this paper we focus on integer lattices, in particular, rank-1 lattice point sets which are defined as follows:
\begin{definition}
Let $N,s\in \NN$ and $\bsz\in \{1,\ldots,N-1\}^s$. A rank-1 lattice point set with $N$ points and generating vector $\bsz$ is defined by
\[ P_{N,\bsz} := \left\{ \left\{ \frac{n\bsz}{N}\right\} \,\big|\, 0\leq n < N \right\} , \]
where $\{x\}=x-\lfloor x\rfloor$ denotes the fractional part of $x\in \RR$ and is applied component-wise for vectors.
A QMC rule using a (rank-1) lattice point set is called a (rank-1) lattice rule.
\end{definition}
\noindent
For weighted Korobov spaces consisting of periodic functions whose Fourier coefficients decay algebraically fast, 
it is well known that there are good generating vectors such that the corresponding lattice rules 
achieve the almost optimal rate of convergence of the worst-case error 
and also hold a good dependence of the worst-case error bound on the dimension $s$ \cite{Dic04,DSWW06,Kuo03,SR02,SW01}.
Here the weights of function spaces play a role in moderating 
the relative importance of different variables or groups of variables \cite{SW98}.

It is much less known, however, whether or not there are good lattice rules 
for function spaces consisting of non-periodic smooth functions.
In \cite{SW01}, it was proven that the shift-averaged worst-case error of randomly shifted lattice rules 
in weighted non-periodic Sobolev spaces of first order smoothness, i.e., function spaces which consist of non-periodic functions 
such that the mixed first partial derivatives are square-integrable, 
coincides with the worst-case error of  (deterministic) lattice rules 
in Korobov spaces with modified weights.
Thereafter it was shown in \cite{Hic02} that there exist good generating vectors such that
the ``randomly shifted and then tent-transformed lattice rules'' achieve the almost optimal order of convergence
$N^{-2+\varepsilon}$, $\varepsilon>0$, of the shift-averaged worst-case error 
in weighted non-periodic Sobolev spaces of second order smoothness.
The point is, however, that the results shown in these papers rely on a random shift of lattice point sets, 
so that the algorithm is not completely deterministic.

Recently, in order to address this issue, Dick, Nuyens and Pillichshammer \cite{DNP14}
introduced so-called weighted half-period cosine spaces which consist of non-periodic smooth functions, 
and together with a successive paper \cite{CKNS16},
it has been proved that the worst-case error of (deterministic) tent-transformed lattice rules in those spaces
is bounded above by the worst-case error of lattice rules in Korobov spaces with modified weights.
This means, there are good deterministic tent-transformed lattice rules which 
achieve the almost optimal rate of convergence of the worst-case error.
Moreover, in \cite{DNP14}, the sum of half-period cosine space and Korobov space was considered 
and symmetrized lattice rules were shown to achieve the almost optimal rate of convergence in this function space
(but with a stronger dependence of the worst-case error bound on the dimension $s$).

We would emphasize, however, that the smoothness of functions in the half-period cosine space is measured
not by the differentiability but by the decay rate of the cosine coefficients of functions.
The only known exception is that the half-period cosine space with the smoothness parameter $\alpha=1$ 
and the Sobolev space of first order smoothness coincides.
Therefore, the measure of smoothness can be equivalently transformed 
from the decay rate of the cosine coefficients to the first order differentiability.
Also for the sum of half-period cosine space and Korobov space, 
it is unknown whether the smoothness of functions can be interpreted in terms of the differentiability.
In fact, the authors of \cite{DNP14} commented
\begin{center}
``We do not know whether $\Hcal(K_{\alpha,\bsgamma,s}^{\rsob})$ differs from $\Hcal(K_{\alpha,\bsgamma,s}^{\rkor+\rcos})$ for $\alpha>1$.''
\end{center}
Here $\Hcal(K_{\alpha,\bsgamma,s}^{\rsob})$ denotes the weighted non-periodic Sobolev space of $\alpha$-th order smoothness, and 
$\Hcal(K_{\alpha,\bsgamma,s}^{\rkor+\rcos})$ does the sum of 
the weighted half-period cosine space and the weighted Korobov space with the common parameter $\alpha$.

\subsection{Summary of main findings}\label{subsec:main_findings}
In the light of the above-mentioned researches, the main contribution of this paper is threefold:
\begin{enumerate}
\item In Section~\ref{sec:embed} we provide two strict subspaces of $\Hcal(K_{\alpha,\bsgamma,s}^{\rsob})$ which are norm equivalent 
to $\Hcal(K_{\alpha,\bsgamma,s}^{\rcos})$ and $\Hcal(K_{\alpha,\bsgamma,s}^{\rkor+\rcos})$, respectively.
Here $\Hcal(K_{\alpha,\bsgamma,s}^{\rcos})$ denotes the weighted half-period cosine space with smoothness parameter $\alpha$.
This implies that $\Hcal(K_{\alpha,\bsgamma,s}^{\rkor+\rcos})$ is strictly smaller than $\Hcal(K_{\alpha,\bsgamma,s}^{\rsob})$ as sets,
which solves the above problem.
\item In Section~\ref{subsec:tent} we prove that the worst-case error of tent-transformed lattice rules in $\Hcal(K_{2,\bsgamma,s}^{\rsob})$ 
is bounded above by the squared worst-case error of lattice rules in Korobov space 
with the smoothness parameter $1$ and modified weights.
\item In Section~\ref{subsec:sym} we consider an intermediate Sobolev space, denoted by $\Hcal(K_{\alpha,\bsgamma,s}^{\rsobodda})$,
between $\Hcal(K_{\alpha,\bsgamma,s}^{\rkor+\rcos})$ and $\Hcal(K_{\alpha,\bsgamma,s}^{\rsob})$,
and show that the worst-case error of symmetrized lattice rules is bounded above by
the squared worst-case error of lattice rules in Korobov space with the smoothness parameter $\alpha/2$ and modified weights.
\end{enumerate}
The latter two results imply that there are good generating vectors such that the tent-transformed and symmetrized lattice rules 
achieve the almost optimal rates of convergence $N^{-2+\varepsilon}$ and $N^{-\alpha+\varepsilon}$ with $\varepsilon>0$, respectively, 
in the corresponding function spaces.
In fact, the fast component-by-component algorithm due to \cite{NC06} is directly available to find such good generating vectors.
We note that the number of function evaluations for symmetrized lattice rules grows exponentially in the dimension $s$,
so that the worst-case error bound depends exponentially on $s$ regardless of the weights $\bsgamma$, 
which does not happen for tent-transformed lattice rules.

Whether or not deterministic tent-transformed lattice rules can achieve $O(N^{-2+\varepsilon})$ convergence 
in $\Hcal(K_{2,\bsgamma,s}^{\rsob})$ has remained unknown for a while after the work of Hickernell \cite{Hic02}. 
It can be seen from our first main result that the results in \cite{CKNS16,DNP14} cannot reach this question.
Our second main result gives an affirmative answer to this question.

\subsection{Basic notation}
Throughout this paper, we denote by $\ZZ$ the set of integers and by $\NN$ the set of positive integers.
We write $\NN_0=\NN\cup \{0\}$, and $1:s =\{1,\ldots,s\}$ for $s\in \NN$.
For a vector $\bsk\in \ZZ^s$ and a subset $u\subseteq 1:s$, we write $\bsk_u=(k_j)_{j\in u}$ and $(\bsk_u,\bszero)=\bsl\in \ZZ^s$ 
where $\ell_j=k_j$ if $j\in u$ and $\ell_j=0$ otherwise.

\section{Reproducing kernel Hilbert spaces}\label{sec:RKHS}
In order to study embeddings and norm equivalences of normed function spaces,
we introduce several reproducing kernel Hilbert spaces (RKHSs) in this section.
Our standard reference on the theory of RKHS is \cite{Aro50}.

\subsection{Korobov spaces}\label{subsec:Korobov}
For $f\colon [0,1]^s\to \RR$ and $\bsh\in \ZZ^s$, the $\bsh$-th Fourier coefficient of $f$ is defined by
\[ \hat{f}(\bsh) := \int_{[0,1]^s}f(\bsx)e^{-2\pi \ri \bsh\cdot \bsx}\rd \bsx, \]
where the dot product $\cdot$ denotes the usual inner product in $\RR^s$.
For a set of weights $\bsgamma=(\gamma_u)_{u\subseteq 1:s}$, $\gamma_u\geq 0$,
the reproducing kernel of the weighted Korobov space $\Hcal(K_{\alpha,\bsgamma,s}^{\rkor})$ with smoothness parameter $\alpha\in \RR$, $\alpha>1/2$, is given by
\[ K_{\alpha,\bsgamma,s}^{\rkor}(\bsx,\bsy) := \sum_{\bsh\in \ZZ^s}r_{\alpha,\bsgamma,s}(\bsh)e^{2\pi \ri \bsh \cdot (\bsx-\bsy)}, \]
where the function $r_{\alpha,\bsgamma,s}\colon \ZZ^s\to \RR$ is defined by $r_{\alpha,\bsgamma,s}(\bszero)=1$ and
\begin{align}\label{eq:weight_function}
r_{\alpha,\bsgamma,s}(\bsh_u,\bszero) = \gamma_u\prod_{j\in u}|h_j|^{-2\alpha} 
\end{align}
for vectors $\bsh_u\in (\ZZ\setminus \{0\})^{|u|}$ with a non-empty subset $u\subseteq 1:s$.
(Here we note that the smoothness parameter $\alpha$ differs by a factor of 2 from 
what has been used in some literature, see \cite{DSWW06,SW01}, so that
one needs to transfer the results given in those papers carefully.)
In particular, when $\alpha$ is a positive integer, it follows from the Fourier series of the Bernoulli polynomial of degree $2\alpha$, denoted by $B_{2\alpha}$, that the reproducing kernel reduces to
\begin{align*}
 K_{\alpha,\bsgamma,s}^{\rkor}(\bsx,\bsy) & = 1+\sum_{\emptyset \neq u\subseteq 1:s}\gamma_u\prod_{j\in u}\sum_{h_j\in \ZZ\setminus \{0\}}\frac{e^{2\pi \ri h_j (x_j-y_j)}}{|h_j|^{2\alpha}} \\
& = 1+\sum_{\emptyset \neq u\subseteq 1:s}\gamma_u\prod_{j\in u}\frac{(2\pi)^{2\alpha}}{(-1)^{\alpha+1}(2\alpha)!}B_{2\alpha}(|x_j-y_j|).
\end{align*}

The inner product in $\Hcal(K_{\alpha,\bsgamma,s}^{\rkor})$ is given by
\begin{align*}
 \langle f,g\rangle_{K_{\alpha,\bsgamma,s}^{\rkor}} & = \sum_{\bsh\in \ZZ^s}\frac{\hat{f}(\bsh)\hat{g}(\bsh)}{r_{\alpha,\bsgamma,s}(\bsh)}\\
& = \hat{f}(\bszero)\hat{g}(\bszero)+\sum_{\emptyset \neq u\subseteq 1:s}\gamma_u^{-1}\sum_{\bsh_u\in (\ZZ\setminus \{0\})^{|u|}}\hat{f}(\bsh_u,\bszero)\hat{g}(\bsh_u,\bszero)\prod_{j\in u}|h_j|^{2\alpha} , 
\end{align*}
where, for $u\subseteq 1:s$ such that $\gamma_u=0$, we assume
\[ \sum_{\bsh_u\in (\ZZ\setminus \{0\})^{|u|}}\hat{f}(\bsh_u,\bszero)\hat{g}(\bsh_u,\bszero)\prod_{j\in u}|h_j|^{2\alpha}=0, \]
and formally set $0/0:=0$.
In what follows, we shall make similar assumptions on the projections for other weighted Hilbert spaces without any further notice.

It is clear from the form of the inner product that the smoothness parameter $\alpha$ 
moderates the decay rate of the Fourier coefficients.
However, the parameter $\alpha$ is also related to the differentiability but only of periodic functions.
For simplicity, let $s=1$ and $\alpha$ be a positive integer. 
Then the squared norm of the space $\Hcal(K_{\alpha,\gamma,1}^{\rkor})$ can be simplified to
\[ \|f\|_{K_{\alpha,\gamma,1}^{\rkor}}^2 = \left( \int_{0}^{1}f(x)\rd x\right)^2 + \frac{1}{(2\pi)^{2\alpha}\gamma}\int_{0}^{1}( f^{(\alpha)}(x))^2 \rd x,\]
where $f^{(\alpha)}$ denotes the $\alpha$-th derivative of $f$, see for instance \cite[Appendix~A]{Novak2008tmp}, 
and $\Hcal(K_{\alpha,\gamma,1}^{\rkor})$ is equivalent to the set
\begin{align*}
\Big\{ f\colon [0,1]\to \RR\, \big|\, & f^{(\tau)}\; \text{absolutely continuous with $f^{(\tau)}(0)=f^{(\tau)}(1)$} \\
& \quad \text{for $\tau=0,\ldots,\alpha-1$},\quad f^{(\alpha)}\in L^2[0,1]\Big\}.
\end{align*}

\subsection{Half-period cosine spaces}
For $f\colon [0,1]^s\to \RR$ and $\bsk\in \NN_0^s$, the $\bsk$-th cosine coefficient of $f$ is defined by
\[ \tilde{f}(\bsk) := \int_{[0,1]^s}f(\bsx)\prod_{j=1}^{s}(2-\delta_{0,k_j})^{1/2}\cos(\pi k_jx_j)\rd \bsx, \]
where $\delta$ denotes the Kronecker delta function.
The reproducing kernel of the weighted half-period cosine space $\Hcal(K_{\alpha,\bsgamma,s}^{\rcos})$ with the smoothness parameter $\alpha\in \RR$, $\alpha>1/2$, introduced in \cite{DNP14} is given by
\[ K_{\alpha,\bsgamma,s}^{\rcos}(\bsx,\bsy) := 1+\sum_{\emptyset \neq u\subseteq 1:s}\sum_{\bsk_u\in \NN^{|u|}}r_{\alpha,\bsgamma,s}(\bsk_u,\bszero)\prod_{j\in u}2\cos(\pi k_jx_j)\cos(\pi k_jy_j), \]
where the function $r_{\alpha,\bsgamma,s}$ is defined as in \eqref{eq:weight_function}.
The inner product in $\Hcal(K_{\alpha,\bsgamma,s}^{\rcos})$ is given by
\[ \langle f,g\rangle_{K_{\alpha,\bsgamma,s}^{\rcos}} = \sum_{\bsk\in \NN_0^s}\frac{\tilde{f}(\bsk)\tilde{g}(\bsk)}{r_{\alpha,\bsgamma,s}(\bsk)}. \]

Importantly the space $\Hcal(K_{\alpha,\bsgamma,s}^{\rcos})$ contains non-periodic functions. 
For instance, as pointed out in \cite{DNP14}, the non-periodic function
\[ x-\frac{1}{2} = -\frac{4}{\pi^2}\sum_{\substack{k=1\\ k\colon \text{odd}}}^{\infty}\frac{\cos(\pi kx)}{k^2}\]
is included in $\Hcal(K_{\alpha,\gamma,1}^{\rcos})$ for $1/2<\alpha<3/2$ but not in $\Hcal(K_{\alpha,\gamma,1}^{\rkor})$ for any $\alpha>1/2$.
However let us remark again that the smoothness parameter $\alpha$ measures 
the decay rate of the cosine coefficients and not the differentiability of functions at this point.

\subsection{Sum of Korobov and half-period cosine spaces}
The sum of the weighted Korobov space and the weighted half-period cosine space, 
denoted by $\Hcal(K_{\alpha,\bsgamma,s}^{\rkor+\rcos})$, also introduced in \cite{DNP14} is defined as follows.
The reproducing kernel is defined by
\[ K_{\alpha,\bsgamma,s}^{\rkor+\rcos}(\bsx,\bsy) := 1+\sum_{\emptyset \neq u\subseteq 1:s}\gamma_u\prod_{j\in u}\left[-1+K_{\alpha,1,1}^{\rkor+\rcos}(x_j,y_j)\right], \]
where we write
\begin{align*}
 K_{\alpha,1,1}^{\rkor+\rcos}(x,y) & = \frac{K_{\alpha,1,1}^{\rkor}(x,y)+K_{\alpha,1,1}^{\rcos}(x,y)}{2} \\
 & = 1+\sum_{k=1}^{\infty}\frac{\cos(2\pi k(x-y))+\cos(\pi kx)\cos(\pi ky)}{k^{2\alpha}}
\end{align*}
for $x,y\in [0,1]$. According to the result shown in \cite[Part~I, Section~6]{Aro50}, the norm in $\Hcal(K_{\alpha,\bsgamma,s}^{\rkor+\rcos})$ is given by
\[ \| f\|^2_{K_{\alpha,\bsgamma,s}^{\rkor+\rcos}} = 
\min_{f=f_{\rkor}+f_{\rcos}}2^s\left( \| f_{\rkor}\|^2_{K_{\alpha,\bsgamma,s}^{\rkor}}+\| f_{\rcos}\|^2_{K_{\alpha,\bsgamma,s}^{\rcos}}\right), \]
where the minimum is taken over all functions $f_{\rkor}\in \Hcal(K_{\alpha,\bsgamma,s}^{\rkor})$ and $f_{\rcos}\in \Hcal(K_{\alpha,\bsgamma,s}^{\rcos})$ such that $f=f_{\rkor}+f_{\rcos}$.

\subsection{Non-periodic Sobolev spaces}
Let $\alpha$ be a positive integer.
The reproducing kernel of the weighted non-periodic Sobolev space $\Hcal(K_{\alpha,\bsgamma,s}^{\rsob})$ 
of $\alpha$-th order smoothness is given by
\[ K_{\alpha,\bsgamma,s}^{\rsob}(\bsx,\bsy) := 1+\sum_{\emptyset \neq u\subseteq 1:s}\gamma_u\prod_{j\in u}\left[ -1+K_{\alpha,1,1}^{\rsob}(x_j,y_j)\right], \]
where we define
\[ K_{\alpha,1,1}^{\rsob}(x,y) := 1+\sum_{\tau=1}^{\alpha}\frac{B_{\tau}(x)B_{\tau}(y)}{(\tau!)^2}+(-1)^{\alpha+1}\frac{B_{2\alpha}(|x-y|)}{(2\alpha)!}.\]
Here $B_{\tau}$ denotes the Bernoulli polynomial of degree $\tau$.
The inner product of $\Hcal(K_{\alpha,\bsgamma,s}^{\rsob})$ is given by
\begin{align*}
\langle f,g\rangle_{K_{\alpha,\bsgamma,s}^{\rsob}} 
& = 1+\sum_{\emptyset \neq u\subseteq 1:s}\gamma_u^{-1}\sum_{v\subseteq u}\sum_{\bstau_{u\setminus v}\in \{1,\ldots,\alpha-1\}^{|u\setminus v|}} \\
& \quad \int_{[0,1]^{|v|}}\left( \int_{[0,1]^{s-|v|}}f^{(\bstau_{u\setminus v},\bsalpha_v,\bszero)}(\bsx)\rd \bsx_{1:s\setminus v} \right) \\
& \qquad \qquad \times \left( \int_{[0,1]^{s-|v|}}g^{(\bstau_{u\setminus v},\bsalpha_v,\bszero)}(\bsx)\rd \bsx_{1:s\setminus v} \right)\rd \bsx_v
\end{align*}
where $(\bstau_{u\setminus v},\bsalpha_v,\bszero)$ denotes the $s$-dimensional vector whose $j$-th component is $\tau_j$ if $j\in u\setminus v$, $\alpha$ if $j\in v$ and $0$ otherwise, and then $f^{(\bstau_{u\setminus v},\bsalpha_v,\bszero)}$ denotes the $(\bstau_{u\setminus v},\bsalpha_v,\bszero)$-th partial mixed derivative of $f$.

In the univariate case, the corresponding squared norm is given by
\[ \|f\|_{K_{\alpha,\gamma,1}^{\rsob}}^2 = \left( \int_{0}^{1}f(x)\rd x\right)^2+\frac{1}{\gamma}\left[\sum_{\tau=1}^{\alpha-1}\left( \int_{0}^{1}f^{(\tau)}(x)\rd x\right)^2 + \int_{0}^{1}( f^{(\alpha)}(x))^2 \rd x\right] ,\]
and $\Hcal(K_{\alpha,\gamma,1}^{\rsob})$ is equivalent to
\begin{align*}
\Big\{ f\colon [0,1]\to \RR\, \big|\, & f^{(\tau)}\; \text{absolutely continuous} \\
& \quad \text{for $\tau=0,\ldots,\alpha-1$}, \quad f^{(\alpha)}\in L^2[0,1]\Big\}.
\end{align*}

\subsection{Subspaces of non-periodic Sobolev spaces}
We also consider the following three subspaces of $\Hcal(K_{\alpha,\bsgamma,s}^{\rsob})$ in this paper.
For $s=1$, a weight $\gamma > 0$ and $\alpha \in \NN$, define
\begin{align*}
& \Hcal(K_{\alpha,\gamma,1}^{\rsobodd}) := 
\{f \in \Hcal(K_{\alpha,\gamma,1}^{\rsob})
\mid f^{(\tau)}(0) = f^{(\tau)}(1) \text{ for all odd $\tau \leq \alpha-1$}\},\\
&\Hcal(K_{\alpha,\gamma,1}^{\rsobodda}) := 
\{f \in \Hcal(K_{\alpha,\gamma,1}^{\rsob})
\mid f^{(\tau)}(0) = f^{(\tau)}(1) \text{ for all odd $\tau \leq \alpha-2$}\},\\
&\Hcal(K_{\alpha,\gamma,1}^{\rsoboddbdry}) := \\
& \qquad \{f \in \Hcal(K_{\alpha,\gamma,1}^{\rsob}) \mid f^{(\tau)}(0) = f^{(\tau)}(1) = 0 \text{ for all odd $\tau \leq \alpha-1$}\}.
\end{align*}
These spaces are RKHSs equipped with the inner products 
inherited from $\Hcal(K_{\alpha,\gamma,1}^{\rsob})$
since they are its closed subspaces.
For instance, the reproducing kernels of the first and second spaces are given by
\[
K_{\alpha,\gamma,1}^{\rsobodd}(x,y) := 1+ \gamma\sum_{\substack{\tau=1 \\ \tau\colon \odd}}^{\alpha}\frac{B_{\tau}(x)B_{\tau}(y)}{(\tau!)^2}+(-1)^{\alpha+1}\gamma\frac{B_{2\alpha}(|x-y|)}{(2\alpha)!}
\]
and
\[
K_{\alpha,\gamma,1}^{\rsobodda}(x,y) := 1+ \gamma\sum_{\substack{\tau=1 \\ \tau\colon \odd\ \text{or}\ \tau=\alpha}}^{\alpha}\frac{B_{\tau}(x)B_{\tau}(y)}{(\tau!)^2}+(-1)^{\alpha+1}\gamma\frac{B_{2\alpha}(|x-y|)}{(2\alpha)!},
\]
respectively. We do not give an explicit formula for the kernel of the third one.

For an arbitrary dimension $s$, a set of the weights 
$\bsgamma=(\gamma_u)_{u\subseteq 1:s}$, $\gamma_u\geq 0$, and $\alpha\in \NN$,
we define the respective reproducing kernel by
\[
K_{\alpha,\bsgamma,s}^{\rsob(\bullet)}(\bsx,\bsy)
:= 1+\sum_{\emptyset \neq u\subseteq 1:s}\gamma_u\prod_{j\in u}\left[ -1+K_{\alpha,1,1}^{\rsob(\bullet)}(x_j,y_j)\right],
\]
for $\bullet\in \{\mathrm{odd},\mathrm{odd}+\alpha, \mathrm{odd-bdry}0 \}$.
The following obviously holds:
\[  \Hcal(K_{1,\bsgamma,s}^{\rsoboddbdry}) = \Hcal(K_{1,\bsgamma,s}^{\rsobodd}) = \Hcal(K_{1,\bsgamma,s}^{\rsobodda}) = \Hcal(K_{1,\bsgamma,s}^{\rsob}), \]
\[ \Hcal(K_{\alpha,\bsgamma,s}^{\rsoboddbdry}) \subset \Hcal(K_{\alpha,\bsgamma,s}^{\rsobodd}) =\Hcal(K_{\alpha,\bsgamma,s}^{\rsobodda}) \subset \Hcal(K_{\alpha,\bsgamma,s}^{\rsob}) \]
for odd $\alpha>1$, and
\[ \Hcal(K_{\alpha,\bsgamma,s}^{\rsoboddbdry}) \subset \Hcal(K_{\alpha,\bsgamma,s}^{\rsobodd}) \subset \Hcal(K_{\alpha,\bsgamma,s}^{\rsobodda}) \subset \Hcal(K_{\alpha,\bsgamma,s}^{\rsob}) \]
for even $\alpha$, where we have an exception when $\alpha=2$ for which $\Hcal(K_{2,\bsgamma,s}^{\rsobodda}) = \Hcal(K_{2,\bsgamma,s}^{\rsob})$ holds.
\section{Embeddings and norm equivalences}\label{sec:embed}
Let $\Hcal(K_1)$ and $\Hcal(K_2)$ be RKHSs.
We say that $\Hcal(K_1)$ is continuously embedded in $\Hcal(K_2)$,
which is denoted by
\[ \Hcal(K_1) \hookrightarrow \Hcal(K_2), \]
if $\Hcal(K_1) \subseteq \Hcal(K_2)$ holds as sets and there exists a constant $C>0$ such that
\[ \|f\|_{K_2} \leq C \|f\|_{K_1} \qquad \text{for all $f \in \Hcal(K_1)$}.\]
If $\Hcal(K_1) \hookrightarrow \Hcal(K_2)$ and $\Hcal(K_2) \hookrightarrow \Hcal(K_1)$ hold,
then we say that  $\Hcal(K_1)$ and $\Hcal(K_2)$ are norm equivalent
and write
\[ \Hcal(K_1) \leftrightharpoons \Hcal(K_2). \]
If there exists a function $f$ such that $f \in \Hcal(K_1)$ and $f\notin \Hcal(K_2)$,
we write 
\[ \Hcal(K_1) \not\subset \Hcal(K_2). \]

\subsection{$\Hcal(K_{\alpha,\bsgamma,s}^{\rcos})$ and $\Hcal(K_{\alpha,\bsgamma,s}^{\rsoboddbdry})$}
This subsection is devoted to prove:
\begin{lemma}\label{lem:equiv-cosspace}
For $\alpha\in \NN$, we have
\begin{align*}
\Hcal(K_{\alpha,\bsgamma,s}^{\rcos})
\leftrightharpoons
\Hcal(K_{\alpha,\bsgamma,s}^{\rsoboddbdry}).
\end{align*}
\end{lemma}

\begin{proof}
It suffices to prove the result for the case $s=1$ and a weight $\gamma>0$.
In the case $\alpha=1$, the result immediately follows from 
the equality $\Hcal(K_{1,\gamma,1}^{\rsoboddbdry}) = \Hcal(K_{1,\gamma,1}^{\rsob})$
and the norm equivalence 
$\Hcal(K_{1,\gamma,1}^{\rcos}) \leftrightharpoons \Hcal(K_{1,\gamma,1}^{\rsob})$
shown in \cite{DNP14}.
Thus let us consider the case $\alpha>1$ in the following.

First we prove $\Hcal(K_{\alpha,\gamma,1}^{\rcos}) \hookrightarrow \Hcal(K_{\alpha,\gamma,1}^{\rsoboddbdry})$.
Let
\[ f(x) = a_0+\sum_{k \in \NN} a_k \sqrt{2}\cos(\pi k x) \]
for a real sequence $(a_k)_{k\in \NN_0}$ such that $f \in \Hcal(K_{\alpha,\gamma,1}^{\rcos})$, 
that is, $(a_k)_{k\in \NN_0}$ satisfies
\[ \| f\|^2_{K_{\alpha,\gamma,1}^{\rcos}} = \langle f,f\rangle_{K_{\alpha,\gamma,1}^{\rcos}} = |a_0|^2+\frac{1}{\gamma}\sum_{k\in \NN}|a_k|^2k^{2\alpha}<\infty. \]
As shown in \cite[Section~3.1]{DNP14}, for an odd integer $\tau$ with $1 \leq \tau \leq \alpha-1$, we have
\[ f^{(\tau)}(x) = \sum_{k \in \NN} (-1)^{(\tau+1)/2} a_k \sqrt{2}(\pi k)^\tau \sin(\pi k x), \]
which is pointwise absolutely convergent since
\begin{align*}
& \sum_{k \in \NN} \left|(-1)^{(\tau+1)/2} a_k \sqrt{2}(\pi k)^\tau \sin(\pi k x)\right|  \\
& \leq \pi^\tau \sqrt{2}\sum_{k \in \NN} \left|a_k\right| k^\tau \leq \pi^{\tau}\sqrt{2}\left( \sum_{k \in \NN} \left|a_k\right|^2 k^{2\alpha} \right)^{1/2}\left(\sum_{k \in \NN} \frac{1}{k^{2\alpha-2\tau}} \right)^{1/2} \\
& \leq \pi^\tau \sqrt{2\gamma\zeta(2\alpha-2\tau)}\| f\|_{K_{\alpha,\gamma,1}^{\rcos}} < \infty,
\end{align*}
where we used Cauchy--Schwarz inequality in the second inequality and $\zeta(\cdot)$ denotes the zeta function.
Hence it follows that $f^{(\tau)}(0) = f^{(\tau)}(1) = 0$.
Further it is known from \cite[Theorem~1]{DNP14} that
$\Hcal(K_{\alpha,\gamma,1}^{\rcos}) \hookrightarrow \Hcal(K_{\alpha,\gamma,1}^{\rsob})$.
Thus we conclude
$\Hcal(K_{\alpha,\gamma,1}^{\rcos}) \hookrightarrow \Hcal(K_{\alpha,\gamma,1}^{\rsoboddbdry})$.

Next we prove $\Hcal(K_{\alpha,\gamma,1}^{\rsoboddbdry}) \hookrightarrow \Hcal(K_{\alpha,\gamma,1}^{\rcos})$.
Let $f \in \Hcal(K_{\alpha,\bsgamma,1}^{\rsoboddbdry})$.
We note that such a smooth function can be always represented by a cosine series.
For any integer $k \geq 1$, applying integration by parts $\alpha$ times while using the boundary condition
 $f^{(\tau)}(0) = f^{(\tau)}(1) = 0$ for odd integers $\tau$ with $1\leq \tau \leq \alpha-1$ gives the $k$-th cosine coefficient
\[
\tilde{f}(k)  = \sqrt{2} \int_0^1 f(x) \cos(\pi k x) \rd x
= \frac{\sqrt{2}}{(-\pi k)^{\alpha}} \int_0^1 f^{(\alpha)}(x) \, \mathrm{cs}_{\alpha}(\pi k x) \rd x,
\]
where we define
\[
\mathrm{cs}_{\alpha}(x) :=
\begin{cases}
(-1)^{(\alpha-1)/2}\sin(x) & \text{for odd $\alpha$},\\
(-1)^{\alpha/2}\cos(x) & \text{for even $\alpha$}.
\end{cases}
\]
Hence
\begin{align*}
\|f\|_{K_{\alpha,\gamma,1}^{\rcos}}^2
&= (\tilde{f}(0))^2 + \frac{1}{\gamma}\sum_{k=1}^\infty k^{2\alpha}(\tilde{f}(k))^2\\
&= \left(\int_0^1 f(x) \rd x\right)^2
 + \frac{1}{\gamma \pi^{2\alpha}}\sum_{k=1}^\infty \left(\sqrt{2} \int_0^1 f^{(\alpha)}(x) \, \mathrm{cs}_{\alpha}(\pi k x) \rd x \right)^2\\
&\leq \left(\int_0^1 f(x) \rd x\right)^2 +\frac{1}{\gamma \pi^{2\alpha}}\int_0^1 (f^{(\alpha)}(x))^2 \rd x \leq (1+\pi^{-2\alpha}) \|f\|_{K_{\alpha,\gamma,1}^{\rsob}}^2,
\end{align*}
where the first inequality follows from Bessel's inequality.
Thus we have shown that
$\Hcal(K_{\alpha,\gamma,1}^{\rsoboddbdry}) \hookrightarrow \Hcal(K_{\alpha,\gamma,1}^{\rcos})$.\smartqed
\end{proof}

\subsection{$\Hcal(K_{\alpha,\bsgamma,s}^{\rkor+\rcos})$ and $\Hcal(K_{\alpha,\bsgamma,s}^{\rsobodd})$}
This subsection is devoted to prove:
\begin{lemma}\label{lem:equiv-korcosspace}
For $\alpha\in \NN$, we have
\begin{align*}
\Hcal(K_{\alpha,\bsgamma,s}^{\rkor+\rcos})
\leftrightharpoons
\Hcal(K_{\alpha,\bsgamma,s}^{\rsobodd}).
\end{align*}
\end{lemma}

\begin{proof}
It suffices to prove the result for the case $s=1$ and a weight $\gamma>0$.
Since the embedding
$\Hcal(K_{\alpha,\gamma,1}^{\rkor+\rcos}) \hookrightarrow \Hcal(K_{\alpha,\gamma,1}^{\rsobodd})$
can be shown in a way similar to the first part of the proof of Lemma~\ref{lem:equiv-cosspace}, we omit the proof.

We prove the converse embedding $\Hcal(K_{\alpha,\gamma,1}^{\rsobodd})\hookrightarrow \Hcal(K_{\alpha,\gamma,1}^{\rkor+\rcos}) $.
Recall
\begin{align*}
K_{\alpha,\gamma,1}^{\rsobodd}(x,y)
&= 1 +\gamma\sum_{\substack{\tau=1 \\ \tau\colon \odd}}^{\alpha}\frac{B_{\tau}(x)B_{\tau}(y)}{(\tau!)^2}
 +(-1)^{\alpha+1}\gamma\frac{B_{2\alpha}(|x-y|)}{(2\alpha)!}\\
&=: 1 + \gamma K_1(x,y) + \gamma K_2(x,y).
\end{align*}
This means that $\Hcal(K_{\alpha,\gamma ,1}^{\rsobodd})$ is given by the sum of $\Hcal(\gamma K_1)$ and $\Hcal(1+\gamma K_2)$.
The latter space is nothing but the Korobov space $\Hcal(K_{\alpha,\gamma(2\pi)^{-2\alpha},1}^{\rkor})$,
see Section~\ref{subsec:Korobov},
and it follows from \cite[Theorem~1]{DNP14} that $\Hcal(1+\gamma K_2)$ is 
continuously embedded into $\Hcal(K_{\alpha,\gamma,1}^{\rkor+\cos})$.
Thus, in order to prove the desired embedding,
it suffices to show
$\Hcal(K_1) \hookrightarrow \Hcal(K_{\alpha,\gamma,1}^{\rkor+\cos})$.

Following \cite[Part~I, Section~3]{Aro50}, $\Hcal(K_1)$ is a RKHS spanned by
\[
\{B_{\tau}(x) \mid \text{$\tau$ is an odd integer with $1 \leq \tau \leq \alpha$}\}
\]
and in particular is finite dimensional.
Thus we only need to show $\Hcal(K_1) \subset \Hcal(K_{\alpha,\gamma,1}^{\rkor+\cos})$ as sets.
Actually we can prove a bit stronger claim that any finite summation of Bernoulli polynomials of odd degrees
is in $\Hcal(K_{\alpha,\gamma,1}^{\rkor+\cos})$.
Let
\[
f(x) = \sum_{k=0}^q a_k \frac{B_{2k+1}(x)}{(2k+1)!} ,
\]
where $q \geq \alpha/2$ is a positive integer and $a_k \in \RR$.
It suffices to show that $f$ is decomposable as $f(x) = g_1(x) + g_2(x)$
with $g_1 \in \Hcal(K_{\alpha,\gamma,1}^{\rkor})$
and $g_2 \in \Hcal(K_{\alpha,\gamma,1}^{\rcos})$.
Let $V := ((-\pi^2(2k+1)^2)^l)_{k,l=0}^q$ be a non-singular Vandermonde matrix.
Define $g_2 \in \Hcal(K_{\alpha,\gamma,1}^{\rcos})$ by
\[
g_2(x) := -\sum_{k=0}^q \frac{c_k}{2} \cos(\pi (2k+1) x),
\]
where $c_k \in \RR$ are determined as the unique solution of the equation
\[
V (c_0, \dots, c_q)^{\top} = (a_0, \dots, a_q)^{\top}.
\]
Now let $g_1 := f-g_2$.
Then it is straightforward to see that $g_1^{(\tau)}(0) = g_1^{(\tau)}(1)$ holds for all integers $0 \leq \tau \leq 2q+1$, 
and that $g_1$ is infinitely differentiable,
which means $g_1\in \Hcal(K_{\alpha,\gamma,1}^{\rkor})$.
Thus a desired decomposition $f = g_1 + g_2$ holds.\smartqed
\end{proof}

\subsection{Summary of embedding results}
Together with \cite[Theorem~1]{DNP14}, here we give a summary on embeddings and norm equivalences 
of the function spaces introduced in the previous section.
\begin{theorem}\label{thm:embedding}
Let $\alpha\in \NN$. For a dimension $s$ and a set of weights $\bsgamma=(\gamma_u)_{u\subseteq 1:s}$, $\gamma_u\geq 0$, the following holds true:
\begin{enumerate}
\item For $\alpha=1$,
\begin{align*}
& \Hcal(K_{1,\bsgamma,s}^{\rcos}) \leftrightharpoons \Hcal(K_{1,\bsgamma,s}^{\rkor+\cos}) \leftrightharpoons \Hcal(K_{1,\bsgamma,s}^{\rsob}) \\
& \quad = \Hcal(K_{1,\bsgamma,s}^{\rsobodd}) = \Hcal(K_{1,\bsgamma,s}^{\rsobodda}) = \Hcal(K_{1,\bsgamma,s}^{\rsoboddbdry}) .
\end{align*}
\item For odd $\alpha>1$,
\begin{align*}
& \Hcal(K_{\alpha,\bsgamma,s}^{\rcos}) \leftrightharpoons \Hcal(K_{\alpha,\bsgamma,s}^{\rsoboddbdry})  \hookrightarrow \Hcal(K_{\alpha,\bsgamma,s}^{\rkor+\cos}) \\
& \quad \leftrightharpoons \Hcal(K_{\alpha,\bsgamma,s}^{\rsobodd}) = \Hcal(K_{\alpha,\bsgamma,s}^{\rsobodda}) \hookrightarrow \Hcal(K_{\alpha,\bsgamma,s}^{\rsob}).
\end{align*}
\item For even $\alpha$,
\begin{align*}
& \Hcal(K_{\alpha,\bsgamma,s}^{\rcos}) \leftrightharpoons \Hcal(K_{\alpha,\bsgamma,s}^{\rsoboddbdry})  \hookrightarrow \Hcal(K_{\alpha,\bsgamma,s}^{\rkor+\cos}) \\
& \quad \leftrightharpoons \Hcal(K_{\alpha,\bsgamma,s}^{\rsobodd}) \hookrightarrow \Hcal(K_{\alpha,\bsgamma,s}^{\rsobodda}) \hookrightarrow \Hcal(K_{\alpha,\bsgamma,s}^{\rsob}),
\end{align*}
wherein $\Hcal(K_{\alpha,\bsgamma,s}^{\rsobodda}) = \Hcal(K_{\alpha,\bsgamma,s}^{\rsob})$ holds when $\alpha=2$.
\end{enumerate}
\end{theorem}
\noindent
For $\alpha>1$, since $\Hcal(K_{\alpha,\bsgamma,s}^{\rsob})\not\subset \Hcal(K_{\alpha,\bsgamma,s}^{\rsobodd}) $ is obvious and 
this theorem shows $\Hcal(K_{\alpha,\bsgamma,s}^{\rkor+\cos}) \leftrightharpoons \Hcal(K_{\alpha,\bsgamma,s}^{\rsobodd})$,
we see that $\Hcal(K_{\alpha,\bsgamma,s}^{\rsob})$ is strictly larger than $\Hcal(K_{\alpha,\bsgamma,s}^{\rkor+\rcos})$, 
which solves an open problem in \cite{DNP14}.

\section{The worst-case error bounds}
Here we study the worst-case error of tent-transformed lattice rules in $\Hcal(K_{2,\bsgamma,s}^{\rsob})$ 
and the worst-case error of symmetrized lattice rules in $\Hcal(K_{\alpha,\bsgamma,s}^{\rsobodda})$ for even $\alpha$.
Since 
$\Hcal(K_{2,\bsgamma,s}^{\rsob})$ and $\Hcal(K_{\alpha,\bsgamma,s}^{\rsobodda})$ are strictly larger than 
$\Hcal(K_{2,\bsgamma,s}^{\rcos})$ and $\Hcal(K_{\alpha,\bsgamma,s}^{\rkor+\rcos})$, respectively, 
our results are stronger than those in \cite{CKNS16,DNP14}.

As already introduced, $B_{\tau}$ denotes the Bernoulli polynomial of degree $\tau\in \NN$, and in what follows, we write $b_{\tau}(\cdot)=B_{\tau}(\cdot)/\tau!$. 
Further, we define $\tilde{B}_{\tau}\colon \RR\to \RR$ by extending $B_{\tau}$ periodically to $\RR$ and write $\tilde{b}_{\tau}(\cdot) = \tilde{B}_{\tau}(\cdot)/\tau!$.
The Fourier series of $b_\tau$ is given by
\begin{align}\label{eq:fourier_ber}
b_{\tau}(x) = -\frac{1}{(2\pi \ri)^\tau}\sum_{k\in \ZZ\setminus \{0\}}\frac{e^{2\pi \ri kx}}{k^\tau} ,
\end{align}
see \cite[Appendix~C]{SJbook}. Here the above series converges only conditionally for $\tau=1$, whereas it converges pointwise absolutely for $\tau\geq 2$.

Moreover we recall that the squared worst-case error of a QMC rule using a point set $P\subset [0,1]^s$ in a RKHS $\Hcal(K)$ with a kernel $K\colon [0,1]^s\times [0,1]^s\to \RR$ is given by
\begin{align}\label{eq:worst-case-error-formula}
& \left(e^{\wor}(P; \Hcal(K))\right)^2 \nonumber \\
& = \int_{[0,1]^{2s}}K(\bsx,\bsy)\rd \bsx\rd \bsy  - \frac{2}{|P|}\sum_{\bsx\in P}\int_{[0,1]^s}K(\bsx,\bsy)\rd \bsy + \frac{1}{|P|^2}\sum_{\bsx,\bsy\in P}K(\bsx,\bsy),
\end{align}
see \cite[Section~6]{SW98}. This equality shall be used in the subsequent analysis.

\subsection{Tent-transformed lattice rules in $\Hcal(K_{2,\bsgamma,s}^{\rsob})$}\label{subsec:tent}
The tent transformation $\phi\colon [0,1]\to [0,1]$ is defined by
\[ \phi(x) := 1-|2x-1|, \]
For a rank-1 lattice point set $P_{N,\bsz}$ with $N\in \NN$ and $\bsz\in \{1,\ldots,N-1\}^s$, 
the tent-transformed rank-1 lattice point set is simply given by
\[ P_{N,\bsz}^{\phi}:=\left\{ \phi( \bsx) \colon \bsx\in P_{N,\bsz} \right\}.\]
where $\phi$ is applied componentwise, i.e., we write $\phi(\bsx)=(\phi(x_1),\ldots,\phi(x_s))$.
A QMC rule using $P_{N,\bsz}^{\phi}$ is called a \emph{tent-transformed (rank-1) lattice rule}.

Since we have
\[ \int_{[0,1]^s}\int_{[0,1]^s}K_{2,\bsgamma,s}^{\rsob}(\bsx,\bsy)\rd \bsx\rd \bsy=1
\quad \text{and}\quad  \int_{[0,1]^s}K_{2,\bsgamma,s}^{\rsob}(\bsx,\bsy)\rd \bsy = 1, \]
for any $\bsx\in [0,1]^s$, it follows from \eqref{eq:worst-case-error-formula} that
\begin{align}\label{eq:worst-case_tent_sob2}
& (e^{\wor}(P^{\phi}_{N,\bsz}; \Hcal(K_{2,\bsgamma,s}^{\rsob})))^2 \notag \\
& = -1+ \frac{1}{N^2}\sum_{\bsx,\bsy\in P_{N,\bsz}}K_{2,\bsgamma,s}^{\rsob}(\phi(\bsx),\phi(\bsy)) \notag \\
& = \frac{1}{N^2}\sum_{\bsx,\bsy\in P_{N,\bsz}}\sum_{\emptyset \neq u\subseteq 1:s}\gamma_u\prod_{j\in u}\left[ -1+K_{2,1,1}^{\rsob}(\phi(x_j),\phi(y_j))\right].
\end{align}
In Appendix~\ref{app:kernel_tent_Fourier}, we prove the following:
\begin{lemma}\label{lem:kernel_tent_Fourier}
For any $x,y\in [0,1]$, we have
\[ -1+K_{2,1,1}^{\rsob}(\phi(x),\phi(y)) = \frac{1}{\pi^4}\sum_{k,\ell =1}^{\infty}c^\phi(k,\ell)\cos(2\pi kx)\cos(2\pi \ell y) \]
where 
\[ c^\phi(k,\ell) = \begin{cases}
\displaystyle \frac{1}{k^4}\left( \frac{58}{3}-\frac{32}{\pi^2 k^2}\right) & k,\ell \colon \odd, k= \ell, \\
\displaystyle \frac{1}{k^2 \ell^2}\left( \frac{52}{3}-\frac{16}{\pi^2 k^2}-\frac{16}{\pi^2 \ell^2}\right) & k,\ell \colon \odd, k\neq \ell, \\
\displaystyle \frac{6}{k^4} & k,\ell \colon \even, k= \ell, \\
\displaystyle \frac{4}{k^2\ell^2} & k,\ell \colon \even, k\neq \ell, \\
 0 & \text{otherwise.}
\end{cases}\]
\end{lemma}

In order to prove our main result in this subsection, we need to introduce the definition of the dual lattice of a rank-1 lattice point set.
\begin{definition}
For $N,s\in \NN$ and $\bsz\in \{1,\ldots,N-1\}^s$, the dual lattice of the rank-1 lattice point set $P_{N,\bsz}$ is defined by
\[ P_{N,\bsz}^{\perp} := \left\{ \bsk\in \ZZ^s\,|\, \bsk\cdot \bsz \equiv 0\pmod N\right\}. \]
For a non-empty subset $u\subseteq 1:s$, we write
\[ P_{N,\bsz}^{\perp,u} = \left\{ \bsk_u\in (\ZZ\setminus \{0\})^{|u|}\,\big|\, (\bsk_u,\bszero)\cdot \bsz \equiv 0\pmod N\right\}. \]
\end{definition}
\noindent
It is straightforward to see that we have the decomposition
\[ P_{N,\bsz}^{\perp}\setminus \{\bszero\} = \bigcup_{\emptyset \neq u\subseteq 1:s} P_{N,\bsz}^{\perp,u}. \]

\begin{lemma}\label{lem:dual_lattice}
Let $N,s\in \NN$ and $\bsz\in \{1,\ldots,N-1\}^s$. For $\bsk\in \ZZ^s$, we have
\[ \frac{1}{N}\sum_{\bsx\in P_{N,\bsz}}e^{2\pi \ri \bsk\cdot \bsx}=\begin{cases}
1 & \text{if $\bsk\in P_{N,\bsz}^{\perp}$,} \\
0 & \text{otherwise.}
\end{cases} \]
\end{lemma}
\begin{proof}
See \cite[Lemma~2.7]{SJbook}.
\end{proof}

We can show a bound on the worst-case error of tent-transformed lattice rules in $\Hcal(K_{2,\bsgamma,s}^{\rsob})$ as follows.
\begin{theorem}\label{thm:worst-case-error_tent}
Let $N,s\in \NN$ and $\bsz\in \{1,\ldots,N-1\}^s$. The squared worst-case error of a tent-transformed rank-1 lattice rule in $\Hcal(K_{2,\bsgamma,s}^{\rsob})$ is bounded by
\begin{align*}
(e^{\wor}(P_{N,\bsz}^{\phi}; \Hcal(K_{2,\bsgamma,s}^{\rsob})))^2 & \leq \sum_{\emptyset\neq u\subseteq 1:s}\gamma'_u\left(\sum_{\bsk_u\in P_{N,\bsz}^{\perp,u}}\prod_{j\in u}\frac{1}{k_j^2}\right)^2 \\
& \leq (e^{\wor}(P_{N,\bsz}; \Hcal(K_{1,\bsgamma'^{1/2},s}^{\rkor})))^4,
\end{align*}
where $\gamma'_u=\gamma_u(c'/(4\pi^{4}))^{|u|}$ with $c'=58/3$ for $\emptyset \neq u\subseteq 1:s$, and we write $\bsgamma'^{1/2}=(\gamma'^{1/2}_u)_{u\subseteq 1:s}$.
\end{theorem}

\begin{proof}
Using \eqref{eq:worst-case_tent_sob2} and Lemma~\ref{lem:kernel_tent_Fourier}, the squared worst-case error is given by
\begin{align*}
& (e^{\wor}(P^{\phi}_{N,\bsz}; \Hcal(K_{2,\bsgamma,s}^{\rsob})))^2 \\
& =  \frac{1}{N^2}\sum_{\bsx,\bsy\in P_{N,\bsz}}\sum_{\emptyset \neq u\subseteq 1:s}\gamma_u\prod_{j\in u}\left[ -1+K_{2,1,1}^{\rsob}(\phi(x_j),\phi(y_j))\right] \\
& =  \sum_{\emptyset\neq u\subseteq 1:s}\frac{\gamma_u}{\pi^{4 |u|}}\sum_{\bsk_u,\bsl_u\in \NN^{|u|}}\frac{1}{N^2}\sum_{\bsx,\bsy\in P_{N,\bsz}} \prod_{j\in u}c^{\phi}(k_j,\ell_j)\cos(2\pi k_jx_j)\cos(2\pi \ell_j y_j) .
\end{align*}
Using the identity $\cos(2\pi kx)=(e^{2\pi \ri kx}+e^{-2\pi \ri kx})/2$ and Lemma~\ref{lem:dual_lattice}, the inner-most sum reduces to
\begin{align*}
& \frac{1}{N^2}\sum_{\bsx,\bsy\in P_{N,\bsz}}\prod_{j\in u}c^{\phi}(k_j,\ell_j)\cos(2\pi k_jx_j)\cos(2\pi \ell_j y_j) \\
& = \frac{1}{4^{|u|}N^2}\sum_{\bsx,\bsy\in P_{N,\bsz}}\sum_{\bssigma_u,\bssigma'_u\in \{\pm 1\}^{|u|}}\prod_{j\in u}c^{\phi}(k_j,\ell_j)e^{2\pi \sigma_jk_jx_j}e^{2\pi \sigma'_j\ell_j y_j} \\
& = \frac{1}{4^{|u|}}\sum_{\bssigma_u,\bssigma'_u\in \{\pm 1\}^{|u|}}\prod_{j\in u}c^{\phi}(k_j,\ell_j)\frac{1}{N^2}\sum_{\bsx\in P_{N,\bsz}}e^{2\pi \ri \bssigma_u(\bsk_u)\cdot \bsx_u}\sum_{\bsy\in P_{N,\bsz}}e^{2\pi \ri \bssigma'_u(\bsl_u)\cdot \bsy_u} \\
& = \frac{1}{4^{|u|}}\sum_{\substack{\bssigma_u,\bssigma'_u\in \{\pm 1\}^{|u|}\\ \bssigma_u(\bsk_u),\bssigma'_u(\bsl_u)\in P_{N,\bsz}^{\perp,u}}}\prod_{j\in u}c^{\phi}(k_j,\ell_j),
\end{align*}
where we have used the notation $\bssigma_u(\bsk_u)=(\sigma_j k_j)_{j\in u}$. Noting that $0\leq c^{\phi}(k,\ell)\leq c'/(k^2 \ell^2)$ for all $k,\ell \in \NN$, we have
\begin{align*}
& (e^{\wor}(P^{\phi}_{N,\bsz}; \Hcal(K_{2,\bsgamma,s}^{\rsob})))^2 \\
& = \sum_{\emptyset\neq u\subseteq 1:s}\frac{\gamma_u}{(4\pi^4)^{|u|}}\sum_{\bsk_u,\bsl_u\in \NN^{|u|}}\sum_{\substack{\bssigma_u,\bssigma'_u\in \{\pm 1\}^{|u|}\\ \bssigma_u(\bsk_u),\bssigma'_u(\bsl_u)\in P_{N,\bsz}^{\perp,u}}}\prod_{j\in u}c^{\phi}(k_j,\ell_j) \\
& = \sum_{\emptyset\neq u\subseteq 1:s}\frac{\gamma_u}{(4\pi^4)^{|u|}}\sum_{\bsk_u,\bsl_u\in P_{N,\bsz}^{\perp,u}}\prod_{j\in u}c^{\phi}(|k_j|,|\ell_j|) \\
& \leq \sum_{\emptyset\neq u\subseteq 1:s}\gamma_u\left(\frac{c'}{4\pi^4}\right)^{|u|}\sum_{\bsk_u,\bsl_u\in P_{N,\bsz}^{\perp,u}}\prod_{j\in u}\frac{1}{k_j^2 \ell_j^2} = \sum_{\emptyset\neq u\subseteq 1:s}\gamma'_u\left(\sum_{\bsk_u\in P_{N,\bsz}^{\perp,u}}\prod_{j\in u}\frac{1}{k_j^2}\right)^2.
\end{align*}
A further upper bound is obtained by using Jensen's inequality:
\begin{align*}
 (e^{\wor}(P^{\phi}_{N,\bsz}; \Hcal(K_{2,\bsgamma,s}^{\rsob})))^2 & \leq \sum_{\emptyset\neq u\subseteq 1:s}\gamma'_u\left(\sum_{\bsk_u\in P_{N,\bsz}^{\perp,u}}\prod_{j\in u}\frac{1}{k_j^2}\right)^2 \\
& \leq \left( \sum_{\emptyset\neq u\subseteq 1:s}\gamma'^{1/2}_u\sum_{\bsk_u\in P_{N,\bsz}^{\perp,u}}\prod_{j\in u}\frac{1}{k_j^2}\right)^2 \\
& = (e^{\wor}(P_{N,\bsz}; \Hcal(K_{1,\bsgamma'^{1/2},s}^{\rkor})))^4,
\end{align*}
where the last equality is well known, see for instance \cite[Proof of Theorem~2]{DSWW06}.
This completes the proof of the theorem. \smartqed
\end{proof}

Using the Fourier series of $b_2$ as shown in \eqref{eq:fourier_ber}, the squared worst-case error $(e^{\wor}(P_{N,\bsz}; \Hcal(K_{1,\bsgamma'^{1/2},s}^{\rkor})))^2$ for a rank-1 lattice point set $P_{N,\bsz}$ has the following computable formula
\begin{align*}
 (e^{\wor}(P_{N,\bsz}; \Hcal(K_{1,\bsgamma'^{1/2},s}^{\rkor})))^2 
& = \sum_{\emptyset\neq u\subseteq 1:s}\gamma'^{1/2}_u\sum_{\bsk_u\in (\ZZ\setminus \{0\})^{|u|}}\frac{1}{N}\sum_{\bsx\in P_{N,\bsz}}\prod_{j\in u}\frac{e^{2\pi \ri k_jx_j}}{k_j^2} \\
& = \frac{1}{N}\sum_{\bsx\in P_{N,\bsz}}\sum_{\emptyset\neq u\subseteq 1:s}\gamma'^{1/2}_u\prod_{j\in u}\sum_{k_j\in \ZZ\setminus \{0\}}\frac{e^{2\pi \ri k_jx_j}}{k_j^2} \\
& = \frac{1}{N}\sum_{\bsx\in P_{N,\bsz}}\sum_{\emptyset \neq u\subset 1:s}(2c'^{1/2})^{|u|}\gamma_u^{1/2}\prod_{j\in u}b_2(x_j). 
\end{align*}
In particular, for product weights $\gamma_u=\prod_{j\in u}\gamma_j$, we have
\begin{align}\label{eq:cbc_tent_criterion}
 (e^{\wor}(P_{N,\bsz}; \Hcal(K_{1,\bsgamma'^{1/2},s}^{\rkor})))^2 = -1+\frac{1}{N}\sum_{\bsx\in P_{N,\bsz}}\prod_{j=1}^{s}\left[ 1+2c'^{1/2}\gamma_j^{1/2}b_2(x_j)\right]. 
\end{align}

A worst-case error bound of a rank-1 lattice rule found by the so-called component-by-component (CBC) algorithm in the weighted Korobov space with general weights was proved, for instance, in \cite[Theorem~5]{DSWW06}, which is therefore also applicable in our setting.

\begin{corollary}\label{cor:tent-lattice}
Let $N$ be a prime. The CBC algorithm using the squared worst-case error $(e^{\wor}(P_{N,\bsz}; \Hcal(K_{1,\bsgamma'^{1/2},s}^{\rkor})))^2$ as a quality criterion can find a generating vector $\bsz\in \{1,\ldots,N-1\}^{s}$ such that
\begin{align*}
e^{\wor}(P_{N,\bsz}^{\phi}; \Hcal(K_{2,\bsgamma,s}^{\rsob})) & \leq (e^{\wor}(P_{N,\bsz}; \Hcal(K_{1,\bsgamma'^{1/2},s}^{\rkor})))^2 \\
 & \leq \frac{1}{(N-1)^{1/\lambda}}\left( \sum_{\emptyset \neq u\subset 1:s}\gamma^{\lambda/2}_u \left( \frac{2^{1-\lambda}c'^{\lambda/2}\zeta(2 \lambda)}{\pi^{2\lambda}}\right)^{|u|} \right)^{1/\lambda}
\end{align*}
holds for any $\lambda\in (1/2,1]$, where $\zeta(\cdot)$ denotes the zeta function.
\end{corollary}

\begin{remark}
The number of points contained in a tent-transformed lattice point set $P_{N,\bsz}^{\phi}$ is at most $N$.
Under certain conditions on the weights, the worst-case error bound depends only polynomially on $s$
or even becomes independent of $s$. We refer to \cite{Dic04,DSWW06,Kuo03} for such results which also apply in our setting.
\end{remark}

\begin{remark}
The fast CBC algorithm using the fast Fourier transform due to Nuyens and Cools \cite{NC06} can be used to search for a good generating vector $\bsz$. For product weights, we need $O(sN\log N)$ arithmetic operations and $O(N)$ memory. Moreover, for product and order dependent (POD) weights, 
\[ \gamma_u = \Gamma_{|u|}\prod_{j\in u}\gamma_j, \]
we need $O(sN\log N+s^2 N)$ arithmetic operations and $O(sN)$ memory \cite{KSS11}.
\end{remark}

\begin{remark}
Calculating the Fourier series given in Lemma~\ref{lem:kernel_tent_Fourier} is the crucial step to obtain the results of this section.
We see that the decay of the Fourier coefficients $c^{\phi}(k,\ell)\in O(1/(k^2 \ell^2))$ leads to the almost optimal order of convergence $O(N^{-2+\varepsilon})$.
In principle, a similar calculation can be done for a general $\alpha> 2$. As can be expected from the proof in Appendix~\ref{app:kernel_tent_Fourier}, however, 
the decay of the Fourier coefficients remains $c^{\phi}(k,\ell)\in O(1/(k^2 \ell^2))$, which means that we cannot improve the order of convergence.
\end{remark}

\subsection{Symmetrized lattice rules in $\Hcal(K_{\alpha,\bsgamma,s}^{\rsobodda})$}\label{subsec:sym}
For $u \subseteq 1:s$
and $\bsx = (x_1, \dots, x_s) \in [0,1]^s$, we define
$\rsym_u(\bsx) = (y_1, \dots, y_s)\in [0,1]^s$
where
$y_j := 1-x_j$ if $j \in u$ and $y_j := x_j$ otherwise.
For a rank-1 lattice point set $P_{N,\bsz}$,
the symmetrization of $P_{N,\bsz}$ is defined as the multiset 
\[ P_{N,\bsz}^{\rsym} := \bigcup_{u \subseteq 1:s} \rsym_u(P_{N,\bsz}). \]
A QMC rule using $P_{N,\bsz}^{\rsym}$ is called a \emph{symmetrized (rank-1) lattice rule}.
The worst-case error of symmetrized lattice rules in $\Hcal(K_{\alpha,\bsgamma,s}^{\rsobodda})$ is bounded as follows.
We give the proof in Appendix~\ref{app:worst-case-error_sym}.
\begin{theorem}\label{thm:worst-case-error_sym}
Let $N,s\in \NN$ and $\bsz\in \{1,\ldots,N-1\}^s$. The squared worst-case error of a symmetrized rank-1 lattice rule in $\Hcal(K_{\alpha,\bsgamma,s}^{\rsobodda})$ for even $\alpha$ is bounded by
\begin{align*}
(e^{\wor}(P_{N,\bsz}^{\rsym}; \Hcal(K_{\alpha,\bsgamma,s}^{\rsobodda})))^2 & \leq \sum_{\emptyset\neq u\subseteq 1:s}\gamma''_u\left(\sum_{\bsk_u\in P_{N,\bsz}^{\perp,u}}\prod_{j\in u}\frac{1}{k_j^\alpha}\right)^2 \\
& \leq (e^{\wor}(P_{N,\bsz}; \Hcal(K_{\alpha/2,\bsgamma''^{1/2},s}^{\rkor})))^4,
\end{align*}
where $\gamma''_u=\gamma_u(3/(2^{2\alpha+1}\pi^{2\alpha}))^{|u|}$ for $\emptyset \neq u\subseteq 1:s$ and $\bsgamma''^{1/2}=(\gamma''^{1/2}_u)_{u\subseteq 1:s}$.
\end{theorem}

Following a similar argument as in Section~\ref{subsec:tent}, the squared worst-case error $(e^{\wor}(P_{N,\bsz}; \Hcal(K_{\alpha/2,\bsgamma''^{1/2},s}^{\rkor})))^2$ can be shown to have the computable formulas:
\begin{align*}
&  (e^{\wor}(P_{N,\bsz}; \Hcal(K_{\alpha/2,\bsgamma''^{1/2},s}^{\rkor})))^2 \\
& = \frac{1}{N}\sum_{\bsx\in P_{N,\bsz}}\sum_{\emptyset \neq u\subset 1:s}\left(\frac{\sqrt{3}}{(-1)^{1+\alpha/2}\sqrt{2}}\right)^{|u|}\gamma_u^{1/2}\prod_{j\in u}b_\alpha(x_j),
\end{align*}
for general weights $\bsgamma$, and
\begin{align*}
& (e^{\wor}(P_{N,\bsz}; \Hcal(K_{\alpha/2,\bsgamma''^{1/2},s}^{\rkor})))^2 \\
& = -1+\frac{1}{N}\sum_{\bsx\in P_{N,\bsz}}\prod_{j=1}^{s}\left[ 1+\frac{\sqrt{3}}{(-1)^{1+\alpha/2}\sqrt{2}}\gamma_j^{1/2}b_\alpha(x_j)\right] ,
\end{align*}
for product weights.

\begin{corollary}
Let $N$ be a prime. The fast CBC algorithm using the squared worst-case error $(e^{\wor}(P_{N,\bsz}; \Hcal(K_{\alpha/2,\bsgamma''^{1/2},s}^{\rkor})))^2$ as a quality criterion can find a generating vector $\bsz\in \{1,\ldots,N-1\}^{s}$ such that
\begin{align*}
& e^{\wor}(P_{N,\bsz}^{\rsym}; \Hcal(K_{\alpha,\bsgamma,s}^{\rsobodda})) \\
& \leq (e^{\wor}(P_{N,\bsz}; \Hcal(K_{\alpha/2,\bsgamma''^{1/2},s}^{\rkor})))^2 \\
& \leq \frac{1}{(N-1)^{1/\lambda}}\left( \sum_{\emptyset \neq u\subset 1:s}\gamma^{\lambda/2}_u \left( \frac{3^{\lambda/2}\zeta(\alpha \lambda)}{2^{(2\alpha+1)\lambda/2-1}\pi^{\alpha\lambda}}\right)^{|u|} \right)^{1/\lambda}
\end{align*}
holds for any $\lambda\in (1/\alpha,1]$.
\end{corollary}

\begin{remark}
As proven in \cite[Lemma~2]{DNP14}, the number of points contained in a symmetrized lattice point set $P_{N,\bsz}^{\rsym}$
is $2^{s-1}(N+1)$ for odd $N$ and $2^{s-1}N+1$ for even $N$. 
Thus, in terms of the number of points, 
the exponential dependence of the worst-case error bound on $s$ cannot be avoided no matter how fast the weights decay.
This is a major drawback of symmetrized lattice rules.
\end{remark}

\section{Numerical results}
First let us consider a simple bi-variate test function 
\[ f(x,y)=\frac{ye^{xy}}{e-2}, \]
whose integral over $[0,1]^2$ is $1$. We approximate $I(f)$ 
by using lattice rules, tent-transformed lattice rules and symmetrized lattice rules.
For all of the three rules, Fibonacci lattice point sets are employed as the underlying nodes.
Note that $f\in \Hcal(K_{\alpha,\bsgamma,2}^{\rsob})$ for any $\alpha\in \NN$, 
but $f\notin \Hcal(K_{\alpha,\bsgamma,2}^{\rsobodda})$ for $\alpha>2$ 
since it does not satisfy the periodic condition on high order derivatives.
Further since $\Hcal(K_{2,\bsgamma,2}^{\rsobodda})=\Hcal(K_{2,\bsgamma,2}^{\rsob})$ holds,
both tent-transformed lattice rules and symmetrized lattice rules are expected to achieve $O(N^{-2+\varepsilon})$ convergence.
This can be confirmed in Figure~\ref{fig:fibonacci}.
The integration error for symmetrized lattice rules converges exactly with order $N^{-2}$,
whereas the integration error for tent-transformed ones behaves periodically but still converges approximately with order $N^{-2}$.
Although a detailed analysis on the periodic behavior is beyond the scope of this paper, this periodicity coincides with whether the number of the underlying nodes is even or odd.

\begin{figure*}
\centering
\includegraphics[width=0.5\textwidth]{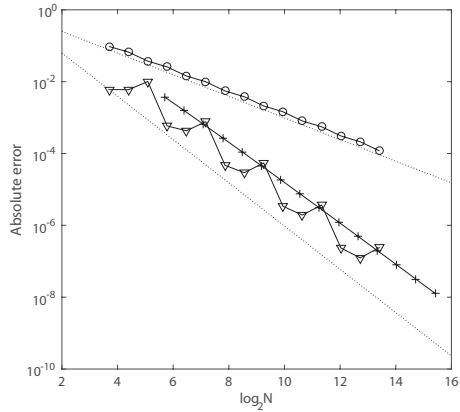}
\caption{The results for $f$ by using lattice rules ($\circ$), tent-transformed lattice rules ($\triangledown$) and symmetrized lattice rules ($+$). The dotted lines correspond to $O(N^{-1})$ and $O(N^{-2})$, respectively.}
\label{fig:fibonacci}
\end{figure*}

Let us move onto the high-dimensional setting. We consider the following test integrands 
\begin{align*}
 f_1(\bsx) & = \prod_{j=1}^{s}\left[ 1+\omega_j\left(x_j^{c_1}-\frac{1}{1+c_1}\right)\right] ,\\
 f_2(\bsx) & = \prod_{j=1}^{s}\left[ 1+\frac{\omega_j}{1+\omega_j x_j^{c_2}} \right] ,
\end{align*}
for the non-negative parameters $c_1,c_2$ and $\omega_j$.
These smooth integrands were originally used in \cite{DGYxx}.
Note that the exact values of the integrals for $f_1$ and for $f_2$ with the special cases $c_2=1$ and $c_2=2$ are known: $I(f_1)=1$ and
\[ I(f_2) = \begin{cases}
\displaystyle \prod_{j=1}^{s}\left[ 1+\log(1+\omega_j)\right] & \text{for $c_2=1$,} \\
\displaystyle \prod_{j=1}^{s}\left[ 1+\sqrt{\omega_j}\tan^{-1}(\sqrt{\omega_j})\right] & \text{for $c_2=2$.}
\end{cases} \]

\begin{figure*}
\centering
\includegraphics[width=0.4\textwidth]{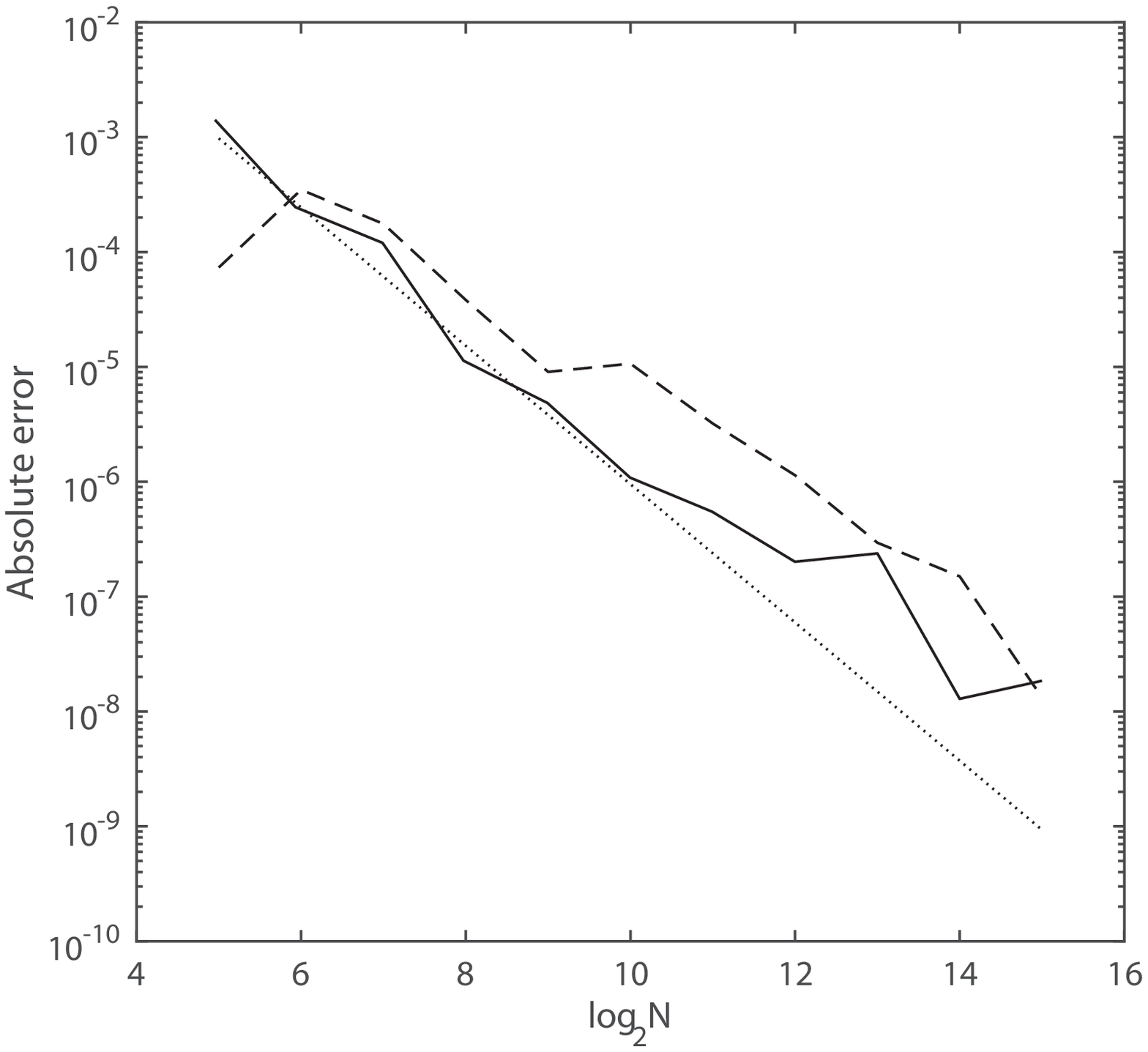}
\includegraphics[width=0.4\textwidth]{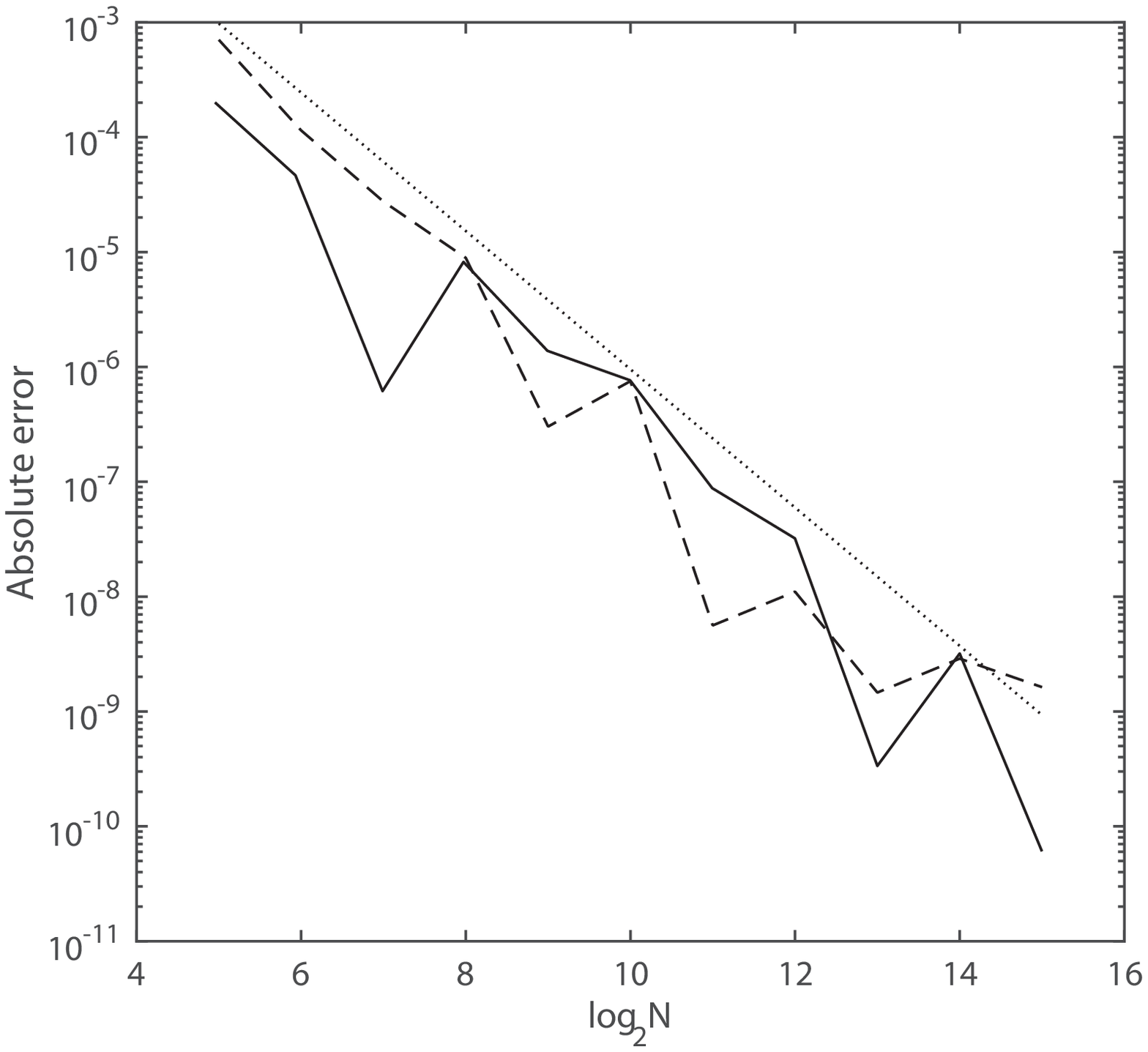}\\
\includegraphics[width=0.4\textwidth]{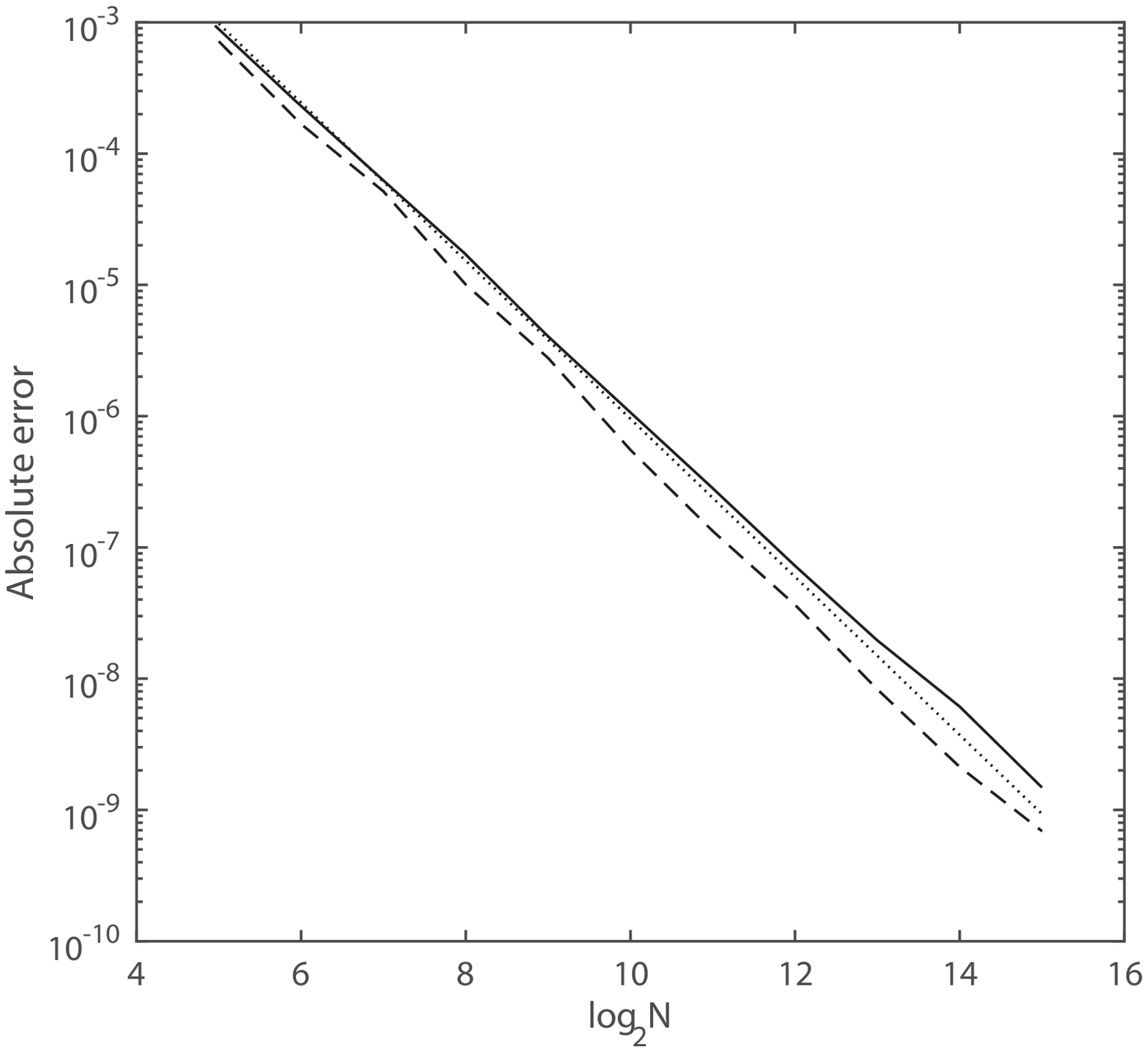}
\includegraphics[width=0.4\textwidth]{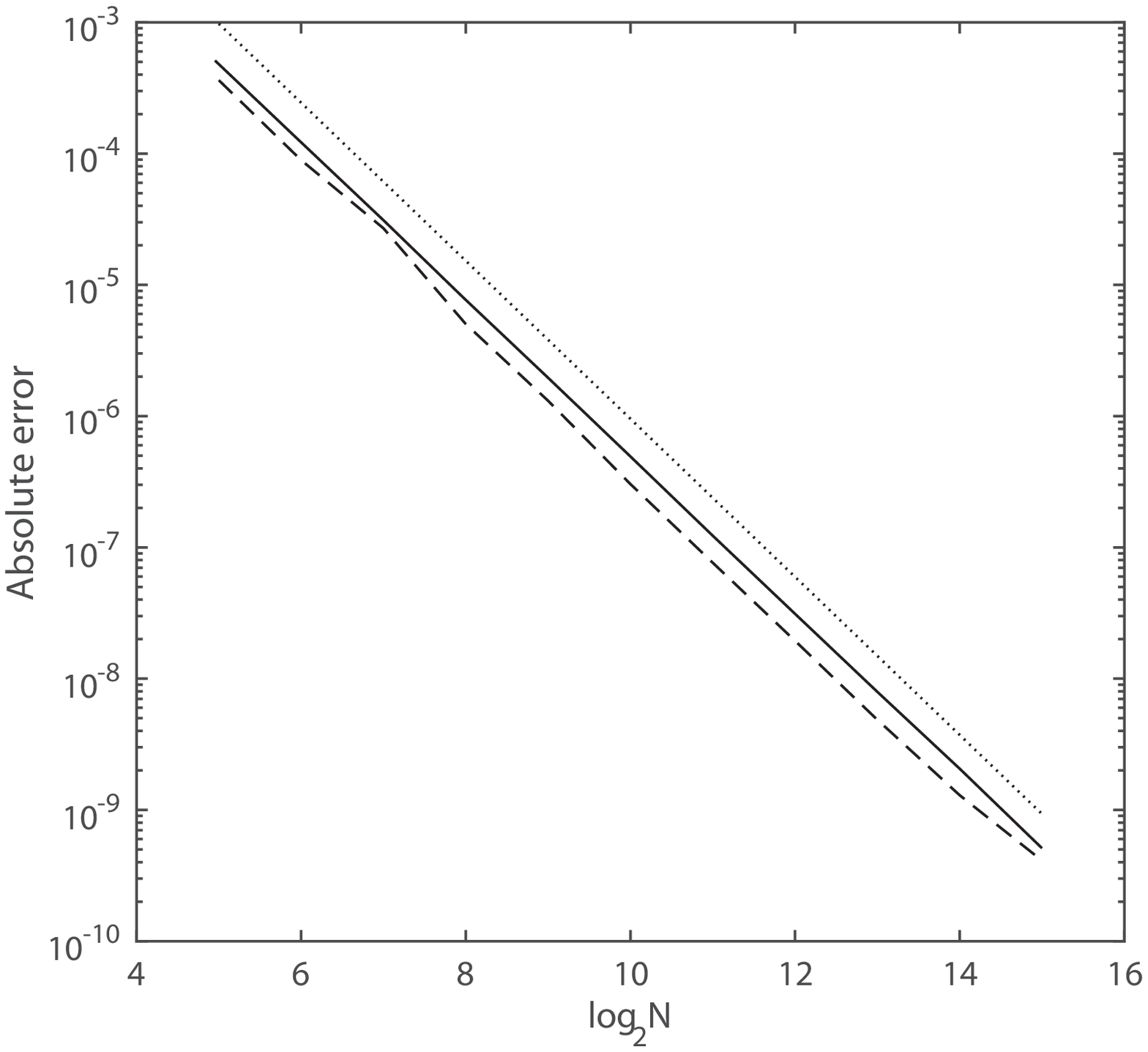}\\
\includegraphics[width=0.4\textwidth]{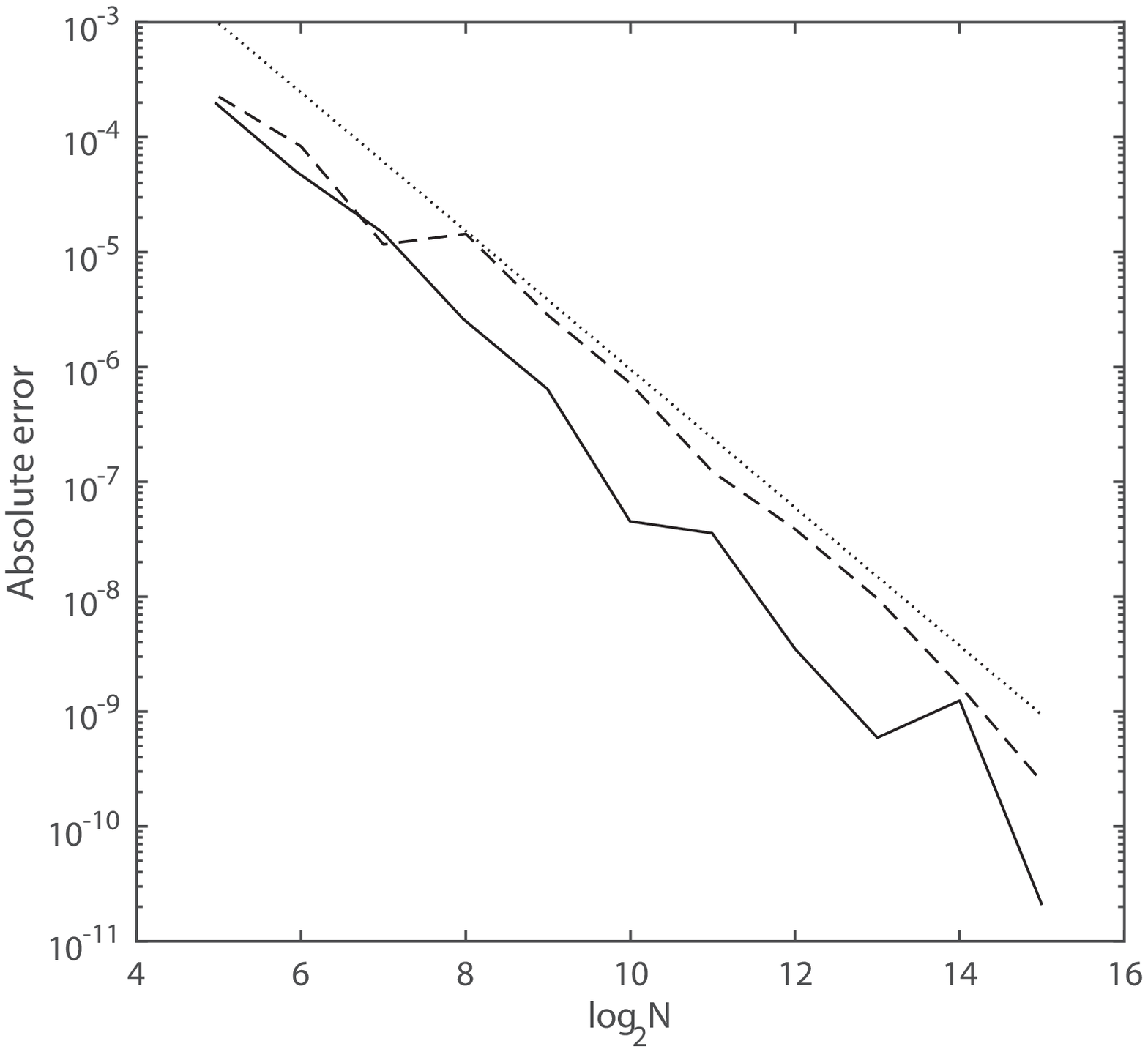}
\includegraphics[width=0.4\textwidth]{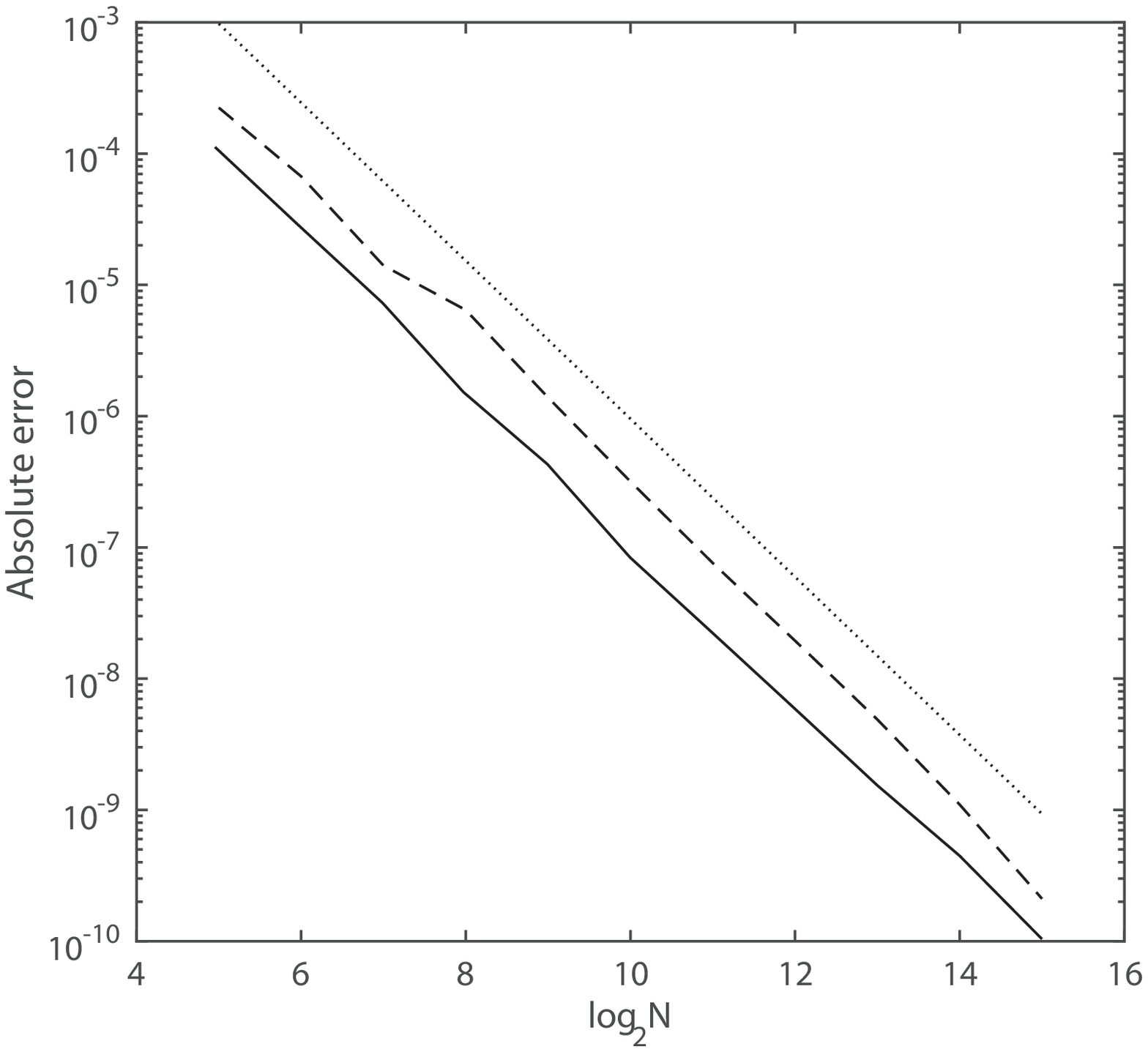}
\caption{The results for $f_1$ with $c_1=1.3$ (top), $f_2$ with $c_2=1$ (middle) and $f_2$ with $c_2=2$ (bottom) by using order 2 Sobol' sequences (dashed) and tent-transformed lattice rules (solid). The left column shows the results for $s=20$ and $\omega_j=1/j^2$, while the right column does for $s=100$ and $\omega_j=1/j^3$. In each graph, the dotted line corresponds to $O(N^{-2})$ convergence.}
\label{fig:func1}
\end{figure*}

We put $s=20$ and $\omega_j=1/j^2$ for the integrands. Since the problem is high-dimensional, 
we focus on tent-transformed lattice rules.
We search for a generating vector using the fast CBC algorithm 
based on the quality criterion \eqref{eq:cbc_tent_criterion} with the choice of the weights $\gamma_j=1/(4c'j^2)$. 
We compare the performance of tent-transformed lattice rules with 
that of order 2 Sobol' sequences as implemented in \cite{Dicweb}.
Here higher order digital nets and sequences introduced in \cite{Dic08} are known 
to achieve high order convergence for non-periodic smooth integrands.
A recent improvement even shows that those nets and sequences achieve 
the optimal order of the error convergence \cite{GSY18},
so that we employ order 2 Sobol' sequences as a competitor for present numerical experiments.
The absolute integration errors as functions of $\log_2 N$ are shown in the left column of Figure~\ref{fig:func1}.
For $f_2$ with $c_2=1$, order 2 Sobol' sequences perform slightly better than tent-transformed lattice rules.
For $f_1$ and $f_2$ with $c_2=2$, however, tent-transformed lattice rules outperform order 2 Sobol' sequences for many $N$.
For any case, the integration error converges approximately with order $N^{-2}$, which supports our theoretical result.

Finally we put $s=100$ and $\omega_j=1/j^3$ for the same integrands. 
We search for a generating vector using the fast CBC algorithm 
based on \eqref{eq:cbc_tent_criterion} with $\gamma_j=1/(4c'j^3)$.
The absolute integration errors as functions of $\log_2 N$ are shown in the right column of Figure~\ref{fig:func1}, 
where we see that tent-transformed lattice rules compare well with order 2 Sobol' sequences.

\begin{acknowledgements}
This work was supported by 
JSPS Grant-in-Aid for Young Scientists No.~15K20964 (T.~G.), 
JSPS Grant-in-Aid for JSPS Fellows No.~17J00466 (K.~S.) and No.~17J02651 (T.~Y.),
and JST CREST.
\end{acknowledgements}

\appendix

\section{Proof of Lemma~\ref{lem:kernel_tent_Fourier}}\label{app:kernel_tent_Fourier}
In order to prove Lemma~\ref{lem:kernel_tent_Fourier}, we need the following result.
\begin{lemma}\label{lem:sin_tent_Fourier}
For $k\in \NN$, we have
\[ \sin(2\pi k\phi(x)) = \frac{8}{\pi}\sum_{\substack{\ell =1\\ \ell\colon \odd}}^{\infty}\frac{k}{4k^2-\ell^2}\cos(2\pi \ell x).\]
\end{lemma}

\begin{proof}
Let $\ell\in \NN_0$. Since $\phi$ is given by
\[ \phi(x)=\begin{cases}
2x & \text{if $x\in [0,1/2]$}, \\
2-2x & \text{otherwise}, \\
\end{cases}\]
it is an easy exercise to check that 
\begin{align*}
 \int_{0}^{1}\sin(2\pi k\phi(x))\sin(2\pi \ell x) \rd x = 0,
\end{align*}
and
\begin{align*}
\int_{0}^{1}\sin(2\pi k\phi(x))\cos(2\pi \ell x) \rd x & = \int_{0}^{1}\sin(2\pi kx)\cos(\pi \ell x)\rd x \\
& = \frac{1}{2}\int_{0}^{1}\left( \sin((2k+\ell)\pi x)+\sin((2k-\ell)\pi x)\right)\rd x \\
& = \begin{cases}
0 & \text{for even $\ell$,} \\
\displaystyle \frac{4}{\pi}\frac{k}{4k^2-\ell^2} & \text{for odd $\ell$.}
\end{cases}
\end{align*}
Thus the Fourier series of $\sin(2\pi k\phi(\cdot))$ is given by
\begin{align*}
\sin(2\pi k\phi(x)) & = 2\sum_{\ell=1}^{\infty}\cos(2\pi \ell x)\int_{0}^{1}\sin(2\pi k\phi(y))\cos(2\pi \ell y) \rd y \\
& = \frac{8}{\pi}\sum_{\substack{\ell=1\\ \ell\colon \odd}}^{\infty}\frac{k}{4k^2-\ell^2}\cos(2\pi \ell x).
\end{align*}
Hence we complete the proof.\smartqed
\end{proof}

Now we are ready to prove Lemma~\ref{lem:kernel_tent_Fourier}.
\begin{proof}[Proof of Lemma~\ref{lem:kernel_tent_Fourier}]
By definition, we have
\begin{align}\label{eq:tent_kernel}
-1+K_{2,1,1}^{\rsob}(\phi(x),\phi(y)) = b_1(\phi(x))b_1(\phi(y))+b_2(\phi(x))b_2(\phi(y))-\tilde{b}_4(\phi(x)-\phi(y)).
\end{align}
Using the Fourier series of a triangle wave provided in \cite[Chapter~1, 1.444]{GRbook}, we have
\[ b_1(\phi(x))= \phi(x)-\frac{1}{2} = -\frac{4}{\pi^2}\sum_{\substack{k=1\\ k\colon \odd}}^{\infty}\frac{\cos(2\pi kx)}{k^2}. \]
Thus the Fourier series of the first term of \eqref{eq:tent_kernel} is given by
\[ b_1(\phi(x))b_1(\phi(y)) = \frac{1}{\pi^4}\sum_{\substack{k,\ell=1\\ k,\ell\colon \odd}}^{\infty}\frac{16}{k^2\ell^2}\cos(2\pi kx)\cos(2\pi \ell y). \]

Using the Fourier series of $b_2$ as shown in \eqref{eq:fourier_ber} and the equality $\cos(2\pi k\phi(x))=\cos(4\pi kx)$ which holds for any $k\in \NN$ and $x\in [0,1]$, we have
\[ b_2(\phi(x))= \frac{1}{2\pi^2}\sum_{k=1}^{\infty}\frac{\cos(2\pi k\phi(x))}{k^2}=\frac{2}{\pi^2}\sum_{\substack{k=1\\ k\colon \even}}^{\infty}\frac{\cos(2\pi kx)}{k^2}. \]
Thus the Fourier series of the second term is given by
\[ b_2(\phi(x))b_2(\phi(y)) = \frac{1}{\pi^4}\sum_{\substack{k,\ell=1\\ k,\ell\colon \even}}^{\infty}\frac{4}{k^2\ell^2}\cos(2\pi kx)\cos(2\pi \ell y). \]

Finally let us consider the third term of \eqref{eq:tent_kernel}. 
Using the Fourier series of $\tilde{b}_4$ and the equality $\cos(2\pi k\phi(x))=\cos(4\pi kx)$, we have
\begin{align*}
 \tilde{b}_4(\phi(x)-\phi(y)) & = \frac{-2}{(2\pi)^4}\sum_{k=1}^{\infty}\frac{\cos(2\pi k(\phi(x)-\phi(y)))}{k^4} \\
 & = -\frac{2}{\pi^4}\sum_{\substack{k=1\\ k\colon \even}}^{\infty}\frac{\cos(2\pi kx)\cos(2\pi ky)}{k^4} -\frac{1}{8\pi^4}\sum_{k=1}^{\infty}\frac{\sin(2\pi k\phi(x))\sin(2\pi k\phi(y))}{k^4}.
\end{align*}
Using Lemma~\ref{lem:sin_tent_Fourier}, the second term of the last expression is given by
\begin{align}
& -\frac{1}{8\pi^4}\sum_{k=1}^{\infty}\frac{\sin(2\pi k\phi(x))\sin(2\pi k\phi(y))}{k^4} \nonumber \\
& =  -\frac{8}{\pi^6}\sum_{k=1}^{\infty}\frac{1}{k^2}\sum_{\substack{\ell,m=1\\ \ell,m\colon \odd}}^{\infty}\frac{1}{(4k^2-\ell^2)(4k^2-m^2)}\cos(2\pi \ell x)\cos(2\pi m y) \nonumber \\
& = -\frac{8}{\pi^6}\sum_{\substack{\ell,m=1\\ \ell,m\colon \odd}}^{\infty}\cos(2\pi \ell x)\cos(2\pi m y)\sum_{k=1}^{\infty}\frac{1}{k^2(4k^2-\ell^2)(4k^2-m^2)}.\label{eq:sum_three_prod}
\end{align}
Noting that the equalities
\begin{align*}
 \sum_{k=1}^{\infty}\frac{1}{4k^2-\ell^2} = \frac{1}{2\ell^2}\quad \text{and}\quad \sum_{k=1}^{\infty}\frac{1}{(4k^2-\ell^2)^2} = \frac{1}{4\ell^2}\left( \frac{\pi^2}{4}-\frac{2}{\ell^2} \right).
\end{align*}
hold for any positive odd integer $\ell$, the inner sum of \eqref{eq:sum_three_prod} can be evaluated as follows. In case of $\ell = m$, we have
\begin{align*}
\sum_{k=1}^{\infty}\frac{1}{k^2(4k^2-\ell^2)^2} & = \frac{1}{\ell^2}\sum_{k=1}^{\infty}\left( \frac{1}{\ell^2}\left(\frac{1}{k^2}-\frac{4}{4k^2-\ell^2} \right) + \frac{4}{(4k^2-\ell^2)^2}\right) \\
& = \frac{1}{\ell^2}\left( \frac{1}{\ell^2}\left(\frac{\pi^2}{6}-\frac{2}{\ell^2} \right) + \frac{1}{\ell^2}\left( \frac{\pi^2}{4}-\frac{2}{\ell^2} \right)\right) = \frac{1}{\ell^4}\left(\frac{5}{12}\pi^2-\frac{4}{\ell^2} \right).
\end{align*}
Otherwise if $\ell\neq m$, we have
\begin{align*}
& \sum_{k=1}^{\infty}\frac{1}{k^2(4k^2-\ell^2)(4k^2-m^2)} \\
& = \frac{1}{\ell^2 m^2}\sum_{k=1}^{\infty}\left( \frac{1}{k^2}+\frac{4}{\ell^2-m^2}\left( \frac{m^2}{4k^2-\ell^2}-\frac{\ell^2}{4k^2-m^2}\right)\right) \\
& = \frac{1}{\ell^2 m^2}\left( \frac{\pi^2}{6}+\frac{4}{\ell^2-m^2}\left( \frac{m^2}{2\ell^2}-\frac{\ell^2}{2m^2}\right)\right) = \frac{1}{\ell^2m^2}\left( \frac{\pi^2}{6}-\frac{2}{\ell^2}-\frac{2}{m^2}\right).
\end{align*}
By substituting these results on the Fourier series into \eqref{eq:tent_kernel}, the result of the lemma follows.\smartqed
\end{proof}

\section{Proof of Theorem~\ref{thm:worst-case-error_sym}}\label{app:worst-case-error_sym}

In what follows, for a function $f$ defined over $[0,1]^2$, we define a function $\rsym[f]$ by
\[ \rsym[f(x,y)] := \frac{f(x,y)+f(1-x,y)+f(x,1-y)+f(1-x,1-y)}{4}. \]

Since we have
\[ \int_{[0,1]^s}\int_{[0,1]^s}K_{\alpha,\bsgamma,s}^{\rsobodda}(\bsx,\bsy)\rd \bsx\rd \bsy=1
\quad \text{and}\quad 
\int_{[0,1]^s}K_{\alpha,\bsgamma,s}^{\rsobodda}(\bsx,\bsy)\rd \bsy = 1, \]
for any $\bsx\in [0,1]^s$, it follows from \eqref{eq:worst-case-error-formula} that
\begin{align}\label{eq:orst-case_sym_odd}
& (e^{\wor}(P_{N,\bsz}^{\rsym};\Hcal(K_{\alpha,\bsgamma,s}^{\rsobodda})))^2 \notag \\
& = -1+ \frac{1}{(2^sN)^2}\sum_{\bsx,\bsy\in P_{N,\bsz}^{\rsym}}K_{\alpha,\bsgamma,s}^{\rsobodda}(\bsx ,\bsy) \notag \\
& = -1+ \frac{1}{N^2}\sum_{\bsx,\bsy\in P_{N,\bsz}}\frac{1}{2^{2s}}\sum_{v,w\subseteq 1:s}K_{\alpha,\bsgamma,s}^{\rsobodda}(\rsym_v(\bsx) ,\rsym_w(\bsy)) \notag \\
& =\frac{1}{N^2}\sum_{\bsx,\bsy\in P_{N,\bsz}}\sum_{\emptyset \neq u\subseteq 1:s}\gamma_u \prod_{j\in u}\left[ -1 + \rsym[K_{\alpha,1,1}^{\rsobodda}(x_j ,y_j)]\right].
\end{align}

As we assume that $\alpha$ is even, we have
\begin{align}\label{eq:sym_kernel}
& -1+\rsym[K_{\alpha,1,1}^{\rsobodda}(x ,y)] \notag \\
& = \rsym\left[\sum_{\substack{\tau=1 \\ \tau\colon \odd}}^{\alpha} b_{\tau}(x)b_{\tau}(y) + b_{\alpha}(x)b_{\alpha}(y) - \tilde{b}_{2\alpha}(x-y) \right] \notag \\
& = \sum_{\substack{\tau=1 \\ \tau\colon \odd}}^{\alpha} \rsym\left[b_{\tau}(x)b_{\tau}(y)\right] + \rsym\left[b_{\alpha}(x)b_{\alpha}(y)\right] - \rsym\left[\tilde{b}_{2\alpha}(x-y) \right].
\end{align}
Since $b_\tau(x) = -b_\tau(1-x)$ for odd $\tau$ 
and $b_\tau(x) = b_\tau(1-x)$ for even $\tau$, we have
\[
\rsym[b_\tau(x)b_\tau(y)] =
\begin{cases}
0 & \text{for odd $\tau$},\\
b_\tau(x)b_\tau(y) & \text{for even $\tau$}.
\end{cases}
\]
and the first term of \eqref{eq:sym_kernel} equals $0$.
Considering the Fourier series of $b_\alpha$ for even $\alpha$:
\[
b_\alpha(x)
= \frac{(-1)^{1+\alpha/2}}{(2\pi)^\alpha} \sum_{k\in \ZZ\setminus \{0\}}\frac{e^{2\pi \ri kx}}{k^\alpha}
= \frac{2 \cdot (-1)^{1+\alpha/2}}{(2\pi)^\alpha} \sum_{k=1}^{\infty}\frac{\cos(2\pi kx)}{k^\alpha},
\]
the Fourier series of the second term of \eqref{eq:sym_kernel} is given by
\begin{align*}
\rsym\left[b_{\alpha}(x)b_{\alpha}(y)\right]
= b_{\alpha}(x)b_{\alpha}(y)
= \frac{4}{(2\pi)^{2\alpha}}\sum_{k,\ell=1}^{\infty}\frac{1}{k^\alpha \ell^\alpha}\cos(2\pi kx)\cos(2\pi \ell y).
\end{align*}
Finally, using the Fourier series of $\tilde{b}_{2\alpha}$, the Fourier series of the third term of \eqref{eq:sym_kernel} is given by
\begin{align*}
\rsym\left[\tilde{b}_{2\alpha}(x-y) \right]
= \frac{-2}{(2\pi)^{2\alpha}} \sum_{k=1}^{\infty}\frac{\rsym\left[\cos(2\pi k(x-y))\right]}{k^{2\alpha}} \\
= \frac{-2}{(2\pi)^{2\alpha}} \sum_{k=1}^{\infty}\frac{\cos(2\pi kx)\cos(2\pi ky)}{k^{2\alpha}}.
\end{align*}
By substituting these results on the Fourier series into \eqref{eq:sym_kernel}, we have
\[ -1+\rsym[K_{\alpha,1,1}^{\rsobodda}(x,y)] = \frac{1}{(2\pi)^{2\alpha}}\sum_{k,\ell =1}^{\infty}c^{\rsym}(k,\ell)\cos(2\pi kx)\cos(2\pi \ell y), \]
where 
\[ c^{\rsym}(k,\ell) = \begin{cases}
\displaystyle \frac{6}{k^{2\alpha}} & k= \ell, \\
\displaystyle \frac{4}{k^\alpha \ell^\alpha} & k\neq \ell.
\end{cases} \]
The rest of the proof follows exactly in the same manner as the proof of Theorem~\ref{thm:worst-case-error_tent}.


\end{document}